\documentclass[a4paper,reqno,11pt, oneside]{amsart}

\usepackage{amssymb}
\usepackage{amstext}
\usepackage{amsmath}
\usepackage{amscd}
\usepackage{amsthm}
\usepackage{amsfonts}
\usepackage{enumerate}
\usepackage{graphicx}
\usepackage{color}
\usepackage{here} 

\makeatletter
  
  \@addtoreset{equation}{section}
\makeatother

\newcommand{\calG}{\mathcal{G}}
\newcommand{\calF}{\mathcal{F}}
\newcommand{\calI}{\mathcal{I}}

\newcommand{\calM}{\mathcal{M}}

\newcommand{\ZZ}{\mathbb{Z}}

\newcommand{\RR}{\mathbb{R}}
\newcommand{\CC}{\mathbb{C}}

\newcommand{\kk}{\Bbbk}

\newcommand{\Hom}{\operatorname{Hom}}

\def\opn#1#2{\def#1{\operatorname{#2}}} 
\opn\conv{conv} \opn\mut{mut} \opn\GL{GL} \opn\cone{cone} \opn\ini{in} \opn\NF{NF}

\def\RPsix{H^*(X_{P_6})}
\def\RPfive{H^*(X_{P_5})}
\newcommand{\Int}{\operatorname{Int}}

\newtheorem{thm}{Theorem}[section]

\newtheorem{lem}[thm]{Lemma}
\newtheorem{prop}[thm]{Proposition}
\newtheorem*{problem}{Cohomological rigidity problem for toric manifolds}

\newtheorem*{conj}{Conjecture}

\theoremstyle{definition}
\newtheorem{defi}[thm]{Definition}
\newtheorem{ex}[thm]{Example}

\theoremstyle{remark}
\newtheorem{rem}[thm]{Remark}

\begin{document}

\title{Cohomological rigidity for toric Fano manifolds of small dimensions or large Picard numbers}
\author{Akihiro Higashitani}
\author{Kazuki Kurimoto}
\author{Mikiya Masuda}

\address[A. Higashitani]{Department of Pure and Applied Mathematics, Graduate School of Information Science and Technology, Osaka University, Suita, Osaka 565-0871, Japan}
\email{higashitani@ist.osaka-u.ac.jp}
\address[K. Kurimoto]{Department of Mathematics, Graduate School of Science, Kyoto Sangyo University, Kyoto 603-8555, Japan}
\email{i1885045@cc.kyoto-su.ac.jp}
\address[M. Masuda]{Department of Mathematics, Graduate School of Science, Osaka City University, Sugimoto, Sumiyoshi-ku, Osaka 558-8585, Japan}
\email{masuda@sci.osaka-cu.ac.jp}


\maketitle

\begin{abstract} 
The cohomological rigidity problem for toric manifolds asks whether toric manifolds are diffeomorphic (or homeomorphic) if their integral cohomology rings are isomorphic. Many affirmative partial solutions to the problem have been obtained and no counterexample is known. 
In this paper, we study the diffeomorphism classification of toric Fano $d$-folds with $d=3,4$ or with Picard number $\ge 2d-2$. In particular, we show that those manifolds except for two toric Fano $4$-folds are diffeomorphic if their integral cohomology rings are isomorphic. The exceptional two toric Fano $4$-folds (their ID numbers are 50 and 57 on a list of {\O}bro) have isomorphic cohomology rings and their total Pontryagin classes are preserved under an isomorphism between their cohomology rings, but we do not know whether they are diffeomorphic or homeomorphic. 
\end{abstract}

\section{Introduction}

\subsection{Cohomological rigidity problem}

As is well-known, integral cohomology ring (as a graded ring) is not a complete invariant to distinguish closed smooth manifolds. However, it becomes a complete invariant if we restrict our concern to a small family $\calF$ of closed smooth manifolds. For instance, this is the case if $\calF$ is the family of closed surfaces. We say that a family $\calF$ of manifolds is \textit{cohomologically rigid} if the integral cohomology rings distinguish the manifolds in $\calF$ up to diffeomorphism (or homeomorphism). 

A toric variety is a normal complex algebraic variety with an algebraic action of a $\CC^*$-torus having an open dense orbit. 
It is well known that there is a one-to-one correspondence between toric varieties and a combinatorial object called fans. 
Therefore, the classification of toric varieties reduces to the classification of fans. 
A toric variety is not necessarily compact or smooth. A compact smooth toric variety, which we call a \emph{toric manifold}, is well studied. 
For instance, its cohomology ring and Chern classes are explicitly described in terms of the associated fan. 
As mentioned above, the classification of toric manifolds \emph{as varieties} reduces to the classification of the associated fans. 
However, the classification of toric manifolds \emph{as smooth manifolds} is unknown. 
Motivated by the diffeomorphism classification of a certain family of toric manifolds (\cite{MP08}), the third author and Dong Youp Suh posed the following naive problem in \cite{MS08}. 

\begin{problem}\label{toi}
Are toric manifolds diffeomorphic (or homeomorphic) if their integral cohomology rings are isomorphic as graded rings? Namely, is the family of toric manifolds cohomologically rigid? 
\end{problem}

No counterexample to the problem is known and many affirmative partial solutions have been obtained {(see \cite{CMS11, BEM+17, HKMS18, Choi15} 
and the reference therein for recent accounts of the problem)}. Among those affirmative solutions, 
Bott manifolds are well studied. A Bott manifold is a toric manifold associated with a spanning fan of a cross-polytope. 
It can be obtained as the total space of an iterated $\CC P^1$-bundle starting with a point. 
It is not completely solved that the family of Bott manifolds is cohomologically rigid but the known results are close to the complete solution (\cite{CMM15}). 
Some partial affirmative solutions are also known for generalized Bott manifolds (\cite{CMS101, CMS102}) although the known results are far from the complete solution. 
Here, a generalized Bott manifold is a toric manifold associated with a spanning fan of the direct sum of simplices (see Definition~\ref{def:directsum} for direct sum of polytopes). Similarly to a Bott manifold, a generalized Bott manifold can be obtained as the total space of an iterated $\CC P^{n_i}$-bundle starting with a point, where each $n_i$ can take any positive integer. 

In this paper, we study the diffeomorphism classification of toric \emph{Fano} $d$-folds with $d=3,4$ or with  Picard numbers $\ge 2d-2$. Our main result (Theorem~\ref{thm:main}) below provides another  affirmative partial solution to the cohomological rigidity problem.

\subsection{Smooth Fano $d$-polytopes and toric Fano $d$-folds}

A \textit{lattice} polytope is a convex polytope $P \subset \RR^d$ all of whose vertices lie in the integer lattice $\ZZ^d$. We say that $P \subset \RR^d$ is a \textit{smooth Fano $d$-polytope} if it is a full-dimensional lattice polytope containing the origin in its interior such that the set of vertices of every facet forms a $\ZZ$-basis of $\ZZ^d$. In particular, smooth Fano $d$-polytopes are simplicial. To a smooth Fano $d$-polytope $P \subset \RR^d$, we can associate a complete nonsingular fan as the \textit{spanning fan} of $P$, where each $i$-dimensional cone in the fan is spanned by the vertices of an $(i-1)$-dimensional face of $P$. 

It is known that the set of smooth Fano $d$-polytopes up to unimodular equivalence one-to-one corresponds to the set of toric Fano $d$-folds up to isomorphism as varieties (\cite{Bat99}). Moreover, it is known that for a fixed $d$, there are only finitely many smooth Fano $d$-polytopes up to unimodular equivalence (see \cite{Bat99}). The number of smooth Fano $d$-polytopes (up to unimodular equivalence) for small values of $d$ is given as follows: 

\begin{center}
\begin{tabular}{c|c|c} 
& the number of & \\ 
dimension & smooth Fano $d$-polytopes & proved in \\ \hline\hline
2 & 5 & \\ \hline
3 & 18 & \cite{Bat91, WW82} \\ \hline
4 & 124 & \cite{Bat99, Sato} \\ \hline
5 & 866 & \cite{Oeb}\\ \hline
6 & 7622 & \cite{Oeb} \\ \hline
7 & 72256 & \cite{Oeb} \\ \hline
8 & 749892 & \cite{Oeb}
\end{tabular}
\end{center}
Indeed, {\O}bro (\cite{Oeb}) provides the algorithm (called \textit{SFP algorithm}), 
which produces a complete list of smooth Fano $d$-polytopes (up to unimodular equivalence) for a given positive integer $d$. 
For $d=2,3,4,5,6$, the database of all smooth Fano $d$-polytopes is open in the following URL: 
\begin{center}
{\tt http://www.grdb.co.uk/forms/toricsmooth}
\end{center}
Each toric Fano $d$-fold (or smooth Fano $d$-polytope) has its ID. 
For example, the ID number of Hirzebruch surface of degree $0$ (resp. $1$) is 4 (resp. 3). 
Toric Fano $3$-folds have ID 6--23 and toric Fano $4$-folds have ID 24--147, and so on. 

One can see that among eighteen toric Fano $3$-folds, there are five Bott manifolds (nine generalized Bott manifolds including Bott manifolds), 
and among one hundred twenty-four toric Fano $4$-folds, there are thirteen Bott manifolds (forty-one generalized Bott manifolds including Bott manifolds). 

\subsection{Results}

As explained above, there are only finitely many smooth Fano $d$-polytopes for a given $d$ and we know their explicit description. 
This finiteness and explicitness are a great advantage to investigate the cohomological rigidity problem for toric Fano $d$-folds. In this paper we prove the following.

\begin{thm} \label{thm:main}
The family $\mathcal{F}$ of toric Fano $d$-folds satisfying one of the following condition:
\begin{enumerate}
\item Picard number $\ge 2d-2$,
\item $d=3$,
\item $d=4$ except for ID numbers 50 and 57, 
\end{enumerate}
is cohomologically rigid. Namely, two toric Fano $d$-folds in the family $\mathcal{F}$ are diffeomorphic if and only if their integral cohomology rings are isomorphic as graded rings. 
\end{thm}

We actually identify which toric Fano $d$-folds in the family $\mathcal{F}$ are diffeomorphic (see Tables~\ref{table:diff3fold} and~\ref{table:diff4fold}). Those diffeomorphic toric Fano $d$-folds are in fact weakly equivariantly diffeomorphic with respect to the restricted actions of the compact subtorus of the $\CC^*$-torus. They also show that the main theorem in \cite{Mas08} is incorrect, see Remark~\ref{rem:diffeolemma} (3) for details. 
Toric Fano $4$-folds with ID numbers 50 and 57 have isomorphic integral cohomology rings and their Pontryagin classes are preserved under an isomorphism between their cohomology rings. We do not know whether they are diffeomorphic or homeomorphic, but they are not weakly equivariantly homeomorphic with respect to the restricted actions of the compact subtorus. 

The Picard number of a toric Fano $d$-fold is the number of the vertices of the associated smooth Fano $d$-polytope minus $d$. 
It is known that smooth Fano $d$-polytopes have at most $3d$ vertices and those with at least $3d-2$ vertices are classified (\cite{AJP, Cas06, Obro08}). 
We use this classification result to verify case (1) in Theorem~\ref{thm:main}. As for cases (2) and (3), we use the database of {\O}bro. 

Our approach to the diffeomorphism classification above consists of two directions: one direction is to check an algebraic condition described in terms of fans for two toric manifolds to be homeomorphic or diffeomorphic (Lemma~\ref{diffeolemma}). The other direction is to use the cohomology rings to distinguish the diffeomorphism classes. It is not difficult to carry out the former direction, but it is quite a task to carry out the latter direction in general. 

The cohomology ring $H^*(X)$ of a toric manifold $X$ is the quotient of a polynomial ring in $b_2(X)$ variables by an ideal, where $b_2(X)$ is the rank of $H^2(X)$, i.e. the Picard number of $X$. In some cases, one can check by an elementary method whether two cohomology rings are isomorphic or not. However, in general, the elementary method requires a formidable computation and does not work well. In order to distinguish our cohomology rings, we pay attention to elements of $H^2(X;R)$ whose $k$-th power vanish, where we take $R=\ZZ$, $\ZZ/2$ or $\ZZ/3$ and $k=2, 3$ or $4$. It turns out that they are useful invariants to distinguish our cohomology rings. We often use Gr\"obner basis to compute those invariants. 

Very recently, motivated by McDuff's question on the uniqueness of toric actions on a monotone symplectic manifold, Y. Cho, E. Lee, S. Park and the third author made the following conjecture and verified it for Fano Bott manifolds (\cite{CLMP20}). 

\begin{conj}[\cite{CLMP20}]
If there is a cohomology ring isomorphism between toric Fano manifolds which preserves their first Chern classes, then they are isomorphic as varieties.
\end{conj}

Based on the classification of cohomology rings of our toric Fano manifolds, we prove 

\begin{thm} \label{thm:main2}
The conjecture above is true for toric Fano $d$-folds with $d=3,4$ or with Picard number $\ge 2d-2$.
\end{thm}

\subsection{Structure of the paper} 
In Section~\ref{sec:prepare}, we briefly recall the theory of toric varieties and the well-known presentation of the cohomology ring of a toric manifold (Proposition~\ref{prop:compute}). 
We introduce the invariants of cohomology rings used to distinguish our cohomology rings. 
We also give a lemma (Lemma~\ref{diffeolemma}) mentioned above to find diffeomorphic or homeomorphic toric manifolds. 
We recall the notion of Gr\"obner basis and normal forms. After those preparations, we prove Theorem~\ref{thm:main}. We will verify cases (1), (2), (3) of Theorem~\ref{thm:main} in Sections~\ref{sec:Picard},~\ref{sec:dim3},~\ref{sec:dim4} respectively. Theorem~\ref{thm:main2} will be proved in Section~\ref{sec:c1}.

\subsection*{Acknowledgements}
A. Higashitani was supported in part by JSPS Grant-in-Aid for Scientific Research 18H01134 
and M. Masuda was supported in part by JSPS Grant-in-Aid for Scientific Research 16K05152. 
This work was partly supported by Osaka City University Advanced Mathematical Institute (MEXT Joint Usage/Research Center on Mathematics and Theoretical Physics).

\section{Preliminaries}\label{sec:prepare}

In this section, we recall the well-known presentation of the cohomology ring of a toric manifold (i.e. a compact smooth toric variety) and give a sufficient condition (Lemma~\ref{diffeolemma}) for two compact smooth toric varieties to be homeomorphic or diffeomorphic. We also introduce naive invariants of cohomology rings used in this paper. We show by an example how to compute those invariants using Gr\"obner basis.

\subsection{Toric manifolds and their cohomology rings} 

We briefly review the theory of toric varieties and refer the reader to \cite{Fulton} or \cite{Oda} for details. 

A toric variety of complex dimension $d$ is a normal algebraic variety $X$ over the complex numbers $\CC$ with an algebraic action of $(\CC^*)^d$ having an open dense orbit. A fan in a lattice $N:=\Hom(\CC^*,(\CC^*)^d) (\cong \ZZ^d)$ is a set of rational strongly convex polyhedral cones in $N\otimes \RR$ such that 
\begin{enumerate}
\item each face of a cone in $\Delta$ is also a cone in $\Delta$;
\item the intersection of two cones in $\Delta$ is a face of each.
\end{enumerate}
The fundamental theorem in the theory of toric varieties says that there is a one-to-one correspondence between toric varieties of complex dimension $d$ and fans in $N$, and that two toric varieties are isomorphic if and only if the corresponding fans are unimodularly equivalent. 

For a fan $\Delta$, we denote the corresponding toric variety by $X(\Delta)$. The toric variety $X(\Delta)$ is compact (or complete) if and only if $\Delta$ is complete, i.e. the union of cones in $\Delta$ covers the entire space $N\otimes\RR$. Moreover, $X(\Delta)$ is smooth (or nonsingular) if and only if $\Delta$ is nonsingular, i.e. for any cone $\sigma$ in $\Delta$, the primitive vectors lying on the $1$-dimensional faces of $\sigma$ form a part of a basis of $N$. 

In the following, we assume that our fan $\Delta$ is complete and nonsingular. Let $\rho_1,\dots,\rho_m$ be $1$-dimensional cones in $\Delta$. We denote the primitive vector lying on $\rho_i$ by $v_i$ and call it a \emph{primitive ray vector}. We denote the set of all primitive ray vectors by $\mathcal{V}(\Delta)$. The set of all cones in $\Delta$ defines an abstract simplicial complex on $[m]=\{1,\dots,m\}$. It is called the \emph{underlying simplicial complex} of $\Delta$ and denoted by $\mathcal{K}(\Delta)$. Indeed, a subset $I$ of $[m]$ is a member of $\mathcal{K}(\Delta)$ if and only if $v_i$'s for $i\in I$ span a cone in $\Delta$. Since $\Delta$ can be recovered from two data $\mathcal{K}(\Delta)$ and $\mathcal{V}(\Delta)$, we may think of $\Delta$ as the pair $(\mathcal{K}(\Delta),\mathcal{V}(\Delta))$. 

Each $1$-dimensional cone $\rho_i$ in the fan $\Delta$ corresponds to an invariant divisor $X_i$ of $X(\Delta)$ and we denote the Poincar\'e dual to $X_i$ by $x_i$. 
Since $X_i$ is of real codimension two, $x_i$ lies in $H^2(X(\Delta))$. With this understanding, we have the following well-known presentation of the cohomology ring of a toric manifold.  Throughout this paper, we will use it without mentioning it.  

\begin{prop}[{\cite[Theorem 5.3.1]{ToricTopology}}]\label{prop:compute}
The cohomology ring of a toric manifold $X(\Delta)$ can be described as follows: 
$$H^*(X(\Delta))=\ZZ[x_1,\ldots,x_m]/\mathcal{J}$$
where $\deg x_i=2$ for $i=1,\dots, m$ and $\mathcal{J}$ is the ideal generated by all 
\begin{enumerate}
\item[{\rm (i)}] $x_{i_1}\cdots x_{i_k}$ \quad for $\{i_1,\dots,i_k\} \not\in \mathcal{K}(\Delta)$;
\item[{\rm (ii)}] $\displaystyle{\sum_{j=1}^m u(v_j) x_j}$ \quad for $u\in \Hom(N,\ZZ)$. 
\end{enumerate}
\end{prop}

\begin{rem} \label{rem:pont}
(1) The total Chern class and the total Pontryagin class of $X(\Delta)$ are respectively given by 
\[
c(X(\Delta))=\prod_{i=1}^m (1+x_i),\qquad p(X(\Delta))=\prod_{i=1}^m (1+x_i^2).
\]

(2) Since the fan $\Delta$ is complete and nonsingular, one can eliminate $d$ variables among $x_1,\dots,x_m$ using the linear relations (ii), so $H^*(X(\Delta))$ is actually the quotient of a polynomial ring in $m-d$ variables by an ideal, where $m-d$ agrees with the rank of $H^2(X(\Delta))$ because $k\ge 2$ in (i). 
\end{rem}

\subsection{Diffeomorphism lemma}

In this subsection, we give a sufficient condition for toric manifolds to be homeomorphic or diffeomorphic. By definition, a toric manifold $X(\Delta)$ of complex dimension $d$ has an algebraic action of $(\CC^*)^d$. Let $S^1$ be the unit circle group of $\CC$. It is known that the orbit space $Q$ of $X(\Delta)$ by the restricted action of $(S^1)^d$ is a $d$-dimensional manifold with corners such that all faces (even $Q$ itself) are contractible. The dual face poset of $Q$ defines a simplicial complex which agrees with the simplicial complex $\mathcal{K}(\Delta)$, to be more precise, if $Q_1,\dots,Q_m$ are the facets of $Q$, then $Q_I:=\bigcap_{i\in I}Q_i\not=\emptyset$ for $I\subset [m]$ if and only if $I\in \mathcal{K}(\Delta)$. The orbit space $Q$ is often a simple polytope. In fact, this is the case when $X(\Delta)$ is projective, more generally when $\mathcal{K}(\Delta)$ is the boundary complex of a simplicial polytope. 

Remember that $\mathcal{V}(\Delta)=\{v_1,\dots,v_m\}$ is the set of primitive ray vectors in $\Delta$. Since $N=\Hom(\CC^*,(\CC^*)^d)=\Hom(S^1,(S^1)^d)$, we may regard $v_i\in N$ as a homomorphism from $S^1$ to $(S^1)^d$. We note that $Q$ is the disjoint union of the interior part $\Int Q_I$ of $Q_I$ over $I\in \mathcal{K}(\Delta)$, where we allow $I=\emptyset$ and $Q_{\emptyset}=Q$. We consider the quotient space 
\begin{equation} \label{eq:canomodel}
X(Q,\mathcal{V}(\Delta)):=Q\times (S^1)^d/\!\sim
\end{equation}
where $(x,g)\sim (y,h)$ if and only if $x=y\in \Int Q_I$ and $gh^{-1}$ belongs to the subtorus generated by circle subgroups $v_i(S^1)$ of $(S^1)^d$ for $i\in I$. The space $X(Q,\mathcal{V}(\Delta))$ is called the \emph{canonical model} of $X(\Delta)$ because $X(Q,\mathcal{V}(\Delta))$ is homeomorphic to $X(\Delta)$ (\cite{DJ91}, \cite[Chapter 7]{ToricTopology}). 

\begin{lem} \label{diffeolemma}
Let $\Delta$ and $\Delta'$ be complete nonsingular fans with the same underlying simplicial complex, i.e. $\mathcal{K}(\Delta)=\mathcal{K}(\Delta')$. If the sets of primitive ray vectors $\mathcal{V}(\Delta)=\{v_1,\dots,v_m\}$ and $\mathcal{V}(\Delta')=\{v_1',\dots,v_m'\}$ agree up to sign (which means $v_i=\pm v_i'$ for each $i=1,\dots,m)$, then $X(\Delta)$ and $X(\Delta')$ are homeomorphic. Moreover, if $\mathcal{K}(\Delta)=\mathcal{K}(\Delta')$ is the boundary complex of a simplicial polytope, then \lq\lq homeomorphic\rq\rq\ above can be improved to \lq\lq diffeomorphic\rq\rq. 
\end{lem}

\begin{proof}
As remarked above, $X(\Delta)$ is homeomorphic to the canonical model $X(Q,\mathcal{V}(\Delta))$. We see from the construction that the canonical model does not depend on the signs of $v_i$'s, so the former statement follows. If $\mathcal{K}(\Delta)=\mathcal{K}(\Delta')$ is the boundary complex of a simplicial polytope, then the orbit space $Q$ is a simple polytope; so the latter statement follows from \cite[Proposition 6.4 (iii)]{Davis14}. 
\end{proof}

\begin{rem} \label{rem:diffeolemma}
(1) The multiplication by $(S^1)^d$ on the second factor of $X(Q,\mathcal{V}(\Delta))$ descends to an action of $(S^1)^d$ on $X(Q,\mathcal{V}(\Delta))$ and Lemma~\ref{diffeolemma} holds in this equivariant setting, see \cite[Chapter 7]{ToricTopology}. 

(2) If two complete nonsingular fans are unimodularly equivalent, then the associated two toric manifolds are not only isomorphic but also weakly equivariantly isomorphic with respect to the actions of $(\CC^*)^d$ as is known. 

(3) It is proved in the paper \cite{Mas08} by the third author that if the equivariant cohomology rings of two toric manifolds $X$ and $X'$ are isomorphic as algebras over $H^*(BT)$, where $T$ is the $\CC^*$-torus acting on $X$ and $X'$, then $v_i=\pm v_i'$ (i.e. the condition in Lemma~\ref{diffeolemma} is satisfied). After that, the author claimed that it implies that $X$ and $X'$ are isomorphic as varieties but this is incorrect. The mistake occurs at the very end of the proof of Theorem 1 in \cite{Mas08} and the author thanks Hiraku Abe for pointing out the mistake. Indeed, we will see in this paper that there are toric Fano manifolds which satisfy the condition in Lemma~\ref{diffeolemma} but they are not isomorphic as varieties. The author also thanks Hiroshi Sato for providing such an example. In order to correct the theorem, we need to take equivariant first Chern class into account.  Namely, if the equivariant cohomology algebra isomorphism between $X$ and $X'$ preserves their equivariant first Chern classes (those are $\sum_{i=1}^m\tau_i$ and $\sum_{i=1}^m\tau_i'$ in \cite{Mas08}), then we can  conclude $v_i=v_i'$ for every $i$ so that $X$ and $X'$ are isomorphic as varieties. 
\end{rem}

\subsection{Invariants of a cohomology ring}

Let $\Delta$ be a complete nonsingular fan of dimension $d$. It is well-known that the number of $i$-dimensional cones in $\Delta$ coincides with the $2i$-th Betti number of the toric manifold $X(\Delta)$ for $0\le i \le d-1$. Therefore, when $\Delta$ is the spanning fan of a simplicial $d$-polytope $P$, the face numbers of $P$ are cohomology invariants. Namely, if toric manifolds associated with simplicial polytopes $P$ and $P'$ have isomorphic cohomology rings, then the face numbers of $P$ and $P'$ coincide. Especially, the number of vertices of $P$ and the number of facets of $P$ are cohomology invariants. We will use this fact without mentioning it throughout this paper. 

Betti numbers are invariants of a cohomology ring but they depend only on its additive structure. We now introduce invariants of a cohomology ring, which depend on its ring structure.

\begin{defi}[$s.v.e.$, $c.v.e.$, $k\text{-}v.e.$]
Let $k\ge 2$ and $R=\ZZ$ or $\ZZ/p$ where $p$ is a prime number. We say that a nonzero element of $H^2(X)\otimes R$ is \emph{ $k$-v.e.} over $R$ if it is primitive and its $k$-th power vanishes in $H^*(X)\otimes R$. 
When $R=\ZZ$, the word \lq\lq over $\ZZ$\rq\rq\ will be omitted. Also, $2$-v.e. will be called \emph{s.v.e.} (square vanishing element) and $3$-v.e. will be called \emph{c.v.e.} (cube vanishing element). 
\end{defi}

\begin{defi}[\textit{maximal basis number}]
Let $V$ be the set of all s.v.e. of $H^*(X)$ and we consider 
\[
\mathcal{B} := \{ S \subset V \mid \text{$S$ is a part of a $\ZZ$-basis of $H^2(X)$} \}.
\]
Clearly there exists an $S_{\rm max} \in \mathcal{B}$ such that
$$
|S| \leq |S_{\rm max}| \quad\text{for $\forall S\in \mathcal{B}$}. 
$$
We call $|S_{\rm max}|$ the \textit{maximal basis number} of $H^*(X)$. 
\end{defi}

The number of $k$-v.e. over $R$ and the maximal basis number are invariants of $H^*(X)$. Note that there may exist infinitely many s.v.e. although the maximal basis number is finite. 

\begin{ex}
We compute s.v.e. and the maximal basis numbers of $H^*(F_0)$ and $H^*(F_1)$, where $F_a$ is the Hirzebruch surface of degree $a$. As is well-known, we have 
\begin{align*}
H^*(F_0) \cong \ZZ[x,y]/(x^2,y^2)\qquad\text{and}\qquad H^*(F_1) \cong \ZZ[x,y]/(x^2,y(y-x)). 
\end{align*}
Then we easily obtain the following table.
\begin{table}[H]
\centering
\begin{tabular}{|c|c|c|} \hline
& s.v.e. (up to sign) & maximal basis number \\ \hline
$F_0$ & $x,y$ & $2$ \\ \hline
$F_1$ & $x,x-2y$ & $1$ \\ \hline
\end{tabular}
\end{table}
The number of s.v.e. are the same but the maximal basis numbers are different. Thus, $H^*(F_0)$ and $H^*(F_1)$ are not isomorphic to each other.
\end{ex}

\subsection{Gr\"obner basis and normal forms}

For the computation of $k$-v.e. of $H^*(M)$, we recall what Gr\"obner basis is. 
We refer the reader to \cite{Stu} for the introduction to Gr\"obner basis. The terminologies not defined in this section for Gr\"obner basis can be found there. 

Let $S=\kk[x_1,\ldots,x_m]$ be a polynomial ring with $m$ variables over a field $\kk$ 
and let $\calM$ be the set of all monomials in $S$. 
We say that a total order $<$ in $\calM$ a \textit{monomial order} on $S$ if it satisfies that 
\begin{itemize}
\item $1<u$ for any $u \in \calM$ with $u \neq 1$, and 
\item $uw<vw$ holds for any $u,v,w \in \calM$ with $u<v$. 
\end{itemize}
Fix a monomial order $<$ on $S$. Given a polynomial $f \in S$, we call the leading monomial with respect to $<$ appearing in $f$ 
the \textit{initial monomial}, denoted by $\ini_<(f)$. 
Given an ideal $I \subset S$, the ideal generated by the initial monomials $f$ in $I$ is called the \textit{initial ideal} of $I$, 
denoted by $\ini_<(I)$. Namely, $$\ini_<(I)=(\ini_<(f) \mid f \in I).$$
For a system of generator $\{g_1,\ldots,g_s\}$ of an ideal $I$, it is not necessarily the case that $\ini_<(I)=(\ini_<(g_1),\ldots,\ini_<(g_s))$ holds. 
We say that $\{g_1,\ldots,g_s\}$ is a \textit{Gr\"obner basis for $I$ with respect to a monomial order $<$} if this equality holds. 
Even if a system of generator $\calG$ of an ideal is not a Gr\"obner basis for an ideal, there is an algorithm, so called \textit{Buchberger algorithm}, 
to compute its Gr\"obner basis from $\calG$ by appending some additional generators to $\calG$. 

Let $I \subset S$ be an ideal and let $\{g_1,\ldots,g_s\} \subset I$ be a system of generator of $I$. 
For $f \in S$ which is not equal to $0$, we can get an equation
\begin{align*}
f = f_1g_1 + \cdots + f_sg_s + f'
\end{align*}
satisfying the following: 
\begin{itemize}
\item $u \notin ( \text{in}_{<}(g_1), \ldots, \text{in}_{<}(g_s) )$ for all monomials $u$ appearing in $f'$ if $f' \neq 0$, and 
\item $\text{in}_{<}(f_ig_i) \leq \text{in}_{<}(f)$ if $f_i \neq 0$. 
\end{itemize}
We call $f'$ a \textit{normal form} of $f$ and write $\NF(f)$.
It is known that $\NF(f)$ is well-defined if $\{g_1,\ldots,g_s\}$ is a Gr\"obner basis for $I$. 

Once we get a Gr\"obner basis of an ideal $I$, we have many advantages. One of such advantages is the following proposition: 
\begin{prop}\label{prop:membership}
Let $I \subset S$ be an ideal and let $\{g_1,\ldots,g_s\}$ be a Gr\"obner basis for $I$ with respect to a monomial oder $<$. 
Given a polynomial $f \in S$, we have the following: 
$$f \in I \quad \Longleftrightarrow \quad \NF(f)=0.$$
\end{prop}

\begin{ex}
Let $X=X_{12}$ be the toric Fano manifold corresponding to ID number 12, which will appear in Section 4, 
and let $P$ be the associated smooth Fano $3$-polytope. Then $P$ has six vertices. 

Let us consider the cohomology ring $H^*(X)$, and demonstrate how to compute s.v.e. of this ring. 
According to the database, the vertices of $P$ are as follows: 
\begin{align*}
v_1=(1,0,0), \; v_2=(0,1,0), \; v_3=(0,0,1), \; v_4=(-1,0,1), \; v_5=(0,1,-1), \; v_6=(0,-1,0). 
\end{align*}

1) First, we compute the defining ideal $\calI$ of the cohomology ring. By Proposition~\ref{prop:compute}, we can compute $\calI$ as follows: 
\begin{align*}
H^*(X) &\cong \ZZ[x_1,\ldots,x_6]/(I_X+J_X), \text{ where }\\
I_X&=(x_1x_4,x_2x_6,x_3x_5) \text{ and } J_X=(x_1-x_4, x_2+x_5-x_6, x_3+x_4-x_5). 
\end{align*}
Note that $I_X$ (resp. $J_X$) corresponds to (i) (resp. (ii)). So, we obtain that 
\begin{align*}
H^*(X) &\cong \ZZ[x_1,\ldots,x_6]/((x_1x_4,x_2x_6,x_3x_5)+(x_1-x_4, x_2+x_5-x_6, x_3+x_4-x_5)) \\
&\cong \ZZ[x,y,z]/(x^2,(z-y)z,(y-x)y). 
\end{align*}
Note that we apply the change of variables $x_4=x$, $x_5=y$, $x_6=z$, and $x_1=x$, $x_2=z-y$ and $x_3=y-x$. 

2) Next, we compute a Gr\"obner basis for the ideal $\calI=(x^2,z(z-y),y(y-x))$. 
As a monomial oder, let $<$ be the graded lexicographic order induced by $x<y<z$. 
Then we can see that $\{x^2,z(z-y), y(y-x)\}$ is a Gr\"obner basis for $\calI$ with respect to $<$. 

3) Finally, we compute s.v.e. of $H^*(X)$. Let $f_1=x^2,f_2=z^2-yz$ and $f_3=y^2-xy$. 

Let $f=ax+by+cz$. Then the normal form of $f^2$ can be computed as follows: 
\begin{align*}
f^2&=a^2x^2+b^2y^2+c^2z^2+2abxy+2acxz+2bcyz \\
&=a^2f_1+c^2f_2+b^2f_3+(2a+b)bxy+2acxz+(2b+c)cyz. 
\end{align*}
Hence, $\NF(f^2)=(2a+b)bxy+2acxz+(2b+c)cyz$. Thus, we can see that $\NF(f^2)=0$ if and only if 
\begin{align*}
\begin{cases}
(2a+b)b=0, \\
2ac=0, \\
(2b+c)c=0. 
\end{cases}
\end{align*}
When $a=0$, we have $b=c=0$. When $a \neq 0$, we have $c=0$ and $(2a+b)b=0$, i.e., $b=0$ or $b=-2a$. Hence, we conclude that 
$\NF(f^2)=0$ if and only if $f=ax$ or $f=ax-2ay$. Therefore, s.v.e. of $H^*(X)$ are $x$ and $x-2y$, 
and its maximal basis number is $1$ since $x$ and $x-2y$ cannot form a $\ZZ$-basis. 

\end{ex}

\begin{ex}
Let us consider the cohomology ring $H^*(X_{24})$, which will appear in Section 5, and demonstrate how to compute s.v.e. of this ring. 
Let $X=X_{24}$ and let $P$ be the associated smooth Fano $4$-polytope. Then $P$ has seven vertices. 
According to the database, the vertices of $P$ are as follows: 
\begin{align*}
&v_1=(1,0,0,0), \;\; v_2=(0,1,0,0), \;\; v_3=(0,0,1,0), \;\; v_4=(0,0,0,1), \\
&v_5=(-1,-1,-1,3), \;\; v_6=(0,0,1,-1), \;\; v_7=(0,0,0,-1). 
\end{align*}

1) First, we compute the defining ideal $\calI$ of the cohomology ring. By Proposition~\ref{prop:compute}, we can compute $\calI$ as follows: 
\begin{align*}
H^*(X) &\cong \ZZ[x_1,\ldots,x_8]/(I_X+J_X), \text{ where }\\
I_X&=(x_1x_2x_3x_5,x_3x_7,x_4x_7,x_4x_6,x_1x_2x_5x_6) \text{ and }\\
J_X&=(x_1-x_5,x_2-x_5,x_3-x_5+x_6,x_4+3x_5-x_6-x_7), \\
&\cong \ZZ[x,y,z]/(x^3(x-y),(x-y)z,(-3x+y+z)z,(-3x+y+z)y,x^3y) \\
&=\ZZ[x,y,z]/(x^4,(x-y)z,(-2y+z)z,(-2x+y)y,x^3y). 
\end{align*}
Note that we apply the change of variables $x_5=x$, $x_6=y$, $x_7=z$, and $x_1=x_2=x$, $x_3=x-y$ and $x_4=-3x+y+z$. 

2) Next, we compute a Gr\"obner basis for the ideal $$\calI=(x^4,(x-y)z,(-2y+z)z,(-3x+y+z)y,x^3y).$$ 
As a monomial order, let $<$ be the graded lexicographic order induced by $x<y<z$. 
A Gr\"obner basis for $I$ with respect to a monomial order $<$ is 
\begin{align}\label{eq:example}
\{x^4,(x-y)z,(-2y+z)z,(-2x+y)y,x^3y,x^2z\}. 
\end{align}
Note that the generator itself does not become a Gr\"obner basis since 
\begin{align*}
\calI \ni (x-y)\cdot (x-y)z - z \cdot (-2x+y)y = x^2z \not\in \text{in}_<(\calI), 
\end{align*}
but it follows from Buchberger algorithm that \eqref{eq:example} is a Gr\"obner basis for $\calI$ with respect to $<$. 

3) Finally, we compute s.v.e. of $H^*(X)$. Let $f=ax+by+cz$. Then $\NF(f^2)=a^2x^2 + 2b(a+b)xy + 2c(a+b+c)xz$. 
Therefore, we conclude that ${\mathrm NF}(f^2)=0$ if and only if $a=b=c=0$. This means that $H^*(X)$ has no s.v.e. 
\end{ex}

\section{The case with large Picard number} \label{sec:Picard}

This section is devoted to verifying the case (1) in Theorem~\ref{thm:main}. 
The Picard number of a smooth Fano $d$-fold associated with a smooth Fano $d$-polytope $P$ is the number of vertices of $P$ minus $d$. 
On the other hand, smooth Fano $d$-polytopes are known to have at most $3d$ vertices and those with $3d$, $3d-1$ or $3d-2$ vertices are classified. 
We prepare some terminology to state those results. 

\begin{defi}[\textit{direct sum}] \label{def:directsum}
Let $P \subset \RR^d$ and $Q \subset \RR^e$ be polytopes. Then 
\begin{align*}
P \oplus Q = \text{conv}( P \times \{ {\bf 0}_e \} \cup \{ {\bf 0}_d \} \times Q) \subset \RR^{d+e}, 
\end{align*}
where ${\bf 0}_d$ (resp. ${\bf 0}_e$) denotes the origin of $\RR^d$ (resp. $\RR^e$), is called the \textit{direct sum} of $P$ and $Q$. 
The direct sum is also called the \textit{free sum} of $P$ and $Q$. 
\end{defi}
\begin{defi}[\textit{skew bipyramid}]
Let $P \subset \RR^d$ be a polytope. Then we call a polytope $B \subset \RR^{d+1}$ 
a \textit{skew bipyramid} (or simply \textit{bipyramid}) over $P$ if $P$ is contained in affine hyperplane $H$ such that
there are two vertices $v$ and $w$ of $B$ which lie on either side of $H$ such that
$B = \text{conv}(\{v,w\} \cup P)$ and the line segment $[v,w]$ meets $P$ in its (relative) interior.
\end{defi}

Throughout this section, $e_1,\dots,e_d$ will denote the standard basis of $\ZZ^d$ and $P_6$ (resp. $P_5$) will denote the hexagon (resp. pentagon) with vertices 
\[
\pm e_1,\ \pm e_2,\ \pm (e_1-e_2)\qquad (\text{resp.}\ e_1,\ \pm e_2,\ \pm (e_1-e_2)).
\]
The classification results on smooth Fano $d$-polytopes with $3d$, $3d-1$ or $3d-2$ vertices are as follows. 

\begin{thm}[{Casagrande \cite{Cas06}}]\label{thm:smoothFano3d}
A smooth Fano $d$-polytope $P$ has at most $3d$ vertices. 
If it does have exactly $3d$ vertices, then $d$ is even 
and $P$ is unimodularly equivalent to $P_6^{\oplus \frac{d}{2}}$. 
\end{thm}
\begin{thm}[{$\O$bro \cite{Obro08}}]\label{thm:smoothFano3d-1}
Let $P$ be a smooth Fano $d$-polytope with exactly $3d-1$ vertices. 
If $d$ is even, then $P$ is unimodularly equivalent to $P_5 \oplus P_6^{\oplus \frac{d-2}{2}}$.
If $d$ is odd, then $P$ is unimodularly equivalent to a bipyramid over $P_6^{\oplus \frac{d-1}{2}}$.
\end{thm}
\begin{thm}[{Assarf-Joswig-Paffenholz \cite{AJP}}]\label{thm:smoothFano3d-2}
Let $P$ be a smooth Fano $d$-polytope with exactly $3d-2$ vertices. 
If $d$ is even, then $P$ is unimodularly equivalent to 
\begin{center}
$
\begin{cases}
(1) ~~ \text{a double bipyramid over} ~P_6^{\oplus \frac{d-2}{2}} &\text{or} \\
(2) ~~ P_5 \oplus P_5 \oplus P_6^{\oplus \frac{d-4}{2}} &\text{or} \\
(3) ~~ \text{DP}(4) \oplus P_6^{\oplus \frac{d-4}{2}}, \\
\end{cases}
$
\end{center}
where $DP(4)$ is the convex hull of $10$ vertices 
$ \pm e_1,\ \pm e_2,\ \pm e_3,\ \pm e_4,\ \pm(e_1+e_2+e_3+e_4)$ in $\RR^4$. 
If $d$ is odd, then $P$ is unimodularly equivalent to a bipyramid over $P_5 \oplus P_6^{\oplus \frac{d-3}{2}}$.
\end{thm}

\begin{rem} \label{rem:facetnumber}
The polytopes in (1), (2), (3) in Theorem~\ref{thm:smoothFano3d-2} have different face numbers. 
Indeed, their facet numbers are respectively $24\cdot 6^{\frac{d-4}{2}}$, $25\cdot 6^{\frac{d-4}{2}}$, $30\cdot 6^{\frac{d-4}{2}}$. 
\end{rem}

\subsection{Cohomology of toric Fano $2$-folds associated to $P_6$ and $P_5$} 
Let $X_{P_6}$ be the toric Fano $2$-fold associated to $P_6$. We number the vertices of $P_6$ 
\[
e_1,\ e_2,\ -e_1+e_2,\ -e_1,\ -e_2,\ e_1-e_2 
\]
from $1$ to $6$ and denote the corresponding elements in $H^2(X_{P_6})$ by $\mu_1,\dots,\mu_6$. Then we have 
\begin{align} \label{eq:HP6}
&H^*(X_{P_6}) \nonumber\\
&\cong \mathbb{Z}[\mu_1,\mu_2,\mu_3,\mu_4,\mu_5,\mu_6] / ( (\mu_1\mu_3, \mu_1\mu_4, \mu_1\mu_5, \mu_2\mu_4, \mu_2\mu_5, \mu_2\mu_6, \mu_3\mu_5, \mu_3\mu_6, \mu_4\mu_6) \nonumber\\
&\hspace{5.5cm} +(\mu_1-\mu_3-\mu_4+\mu_6, \mu_2+\mu_3-\mu_5-\mu_6)) \nonumber\\
&\cong \mathbb{Z}[x,y,z,w]/(x(x+y),y(x+y),z(y-w),y(x-z),z(z+w),w(z+w),xz,xw,yw),
\end{align}
where $x=\mu_3$, $y=\mu_4$, $z=\mu_5$ and $w=\mu_6$. 

\begin{lem}\label{lem:MBN_P6_1}
The maximal basis number of $\RPsix$ is $3$. 
\end{lem}

\begin{proof}
It follows from \eqref{eq:HP6} that any element of $H^2(X_{P_6})$ is of the form $ax+by+cz+dw$ with integers $a,b,c,d$ and an elementary computation shows that 
$$ (ax+by+cz+dw)^2=(a^2-2ab+b^2-2bc+c^2-2cd+d^2)w^2. $$
Therefore, the s.v.e of $\RPsix$ are primitive elements in the set 
\begin{equation} \label{eq:sveP6}
\{ ax+by+cz+dw \mid (a-b)^2+(c-d)^2=2bc \}.
\end{equation}
In particular, $x+y$, $y+z$ and $z+w$ are s.v.e. and since they form a part of a $\ZZ$-basis of $H^2(X_{P_6})$, the maximal basis number of $\RPsix$ is at least $3$. 

On the other hand, it follows from \eqref{eq:sveP6} that s.v.e. of $\RPsix$ over $\ZZ/2$ are given by 
\[
\{ ax+by+cz+dw \mid a+b+c+d=0\},
\]
where $a,b,c,d$ are regarded as elements of $\ZZ/2$. Therefore, the dimension of $H^2(X_{P_6})\otimes\ZZ/2$ is $3$, which implies that the maximal basis number of $\RPsix$ is at most 3, proving the lemma. 
\end{proof}

\begin{lem}\label{lem:MBN_P6_2}
For any s.v.e. $f$ of $\RPsix$, there exist infinitely many s.v.e. of $\RPsix$ 
such that $f$ together with the s.v.e. does not form a part of a $\mathbb{Z}$-basis of $H^2(X_{P_6})$. 
\end{lem}
\begin{proof}
As observed in \eqref{eq:sveP6}, any s.v.e. $ax+by+cz+dw$ of $\RPsix$ must satisfy the condition $$(a-b)^2+(c-d)^2=2bc.$$ 
Therefore, the parity of $a-b$ and $c-d$ must be the same. Moreover, at least one of $a,b,c,d$ must be odd
because the s.v.e. $ax+by+cz+dw$ is primitive. In fact, one can see from the identity above that 
the parity of $(a,b,c,d)$ must be one of the following up to symmetry: 
\begin{center}
(i) \;\; (even, \; even, \; odd, \; odd); \quad\quad (ii) \;\; (even, \; odd, \; odd, \; even). 
\end{center}
Since any two elements with the same parity do not form a part of a $\ZZ$-basis of $H^2(X_{P_6})$, 
it suffices to show that there are infinitely many s.v.e. in each of (i) and (ii) above and here are examples of infinitely many s.v.e. in each case: 
\begin{align*}
{\rm (i)} \;\; (2k(k+1),\ 2k^2,\ 1,\ 1), \quad {\rm (ii)} \;\; (2(k+1)^2,\ 2k(k+1)+1,\ 1,\ 0), 
\end{align*}
where $k$ is any integer. 
\end{proof}

Let $X_{P_5}$ be the toric Fano $2$-fold associated to $P_5$. We number the vertices of $P_5$ 
\[
e_1,\ e_2,\ -e_1+e_2,\ -e_2,\ e_1-e_2 
\]
from $1$ to $5$ and denote the corresponding elements in $H^2(X_{P_5})$ by $\mu_1,\dots,\mu_5$. Then we have 
\begin{align} \label{eq:HP5}
H^*(X_{P_5}) \nonumber 
&\cong \mathbb{Z}[\mu_1,\mu_2,\mu_3,\mu_4,\mu_5] / ((\mu_1\mu_3, \mu_1\mu_4, \mu_2\mu_4, \mu_2\mu_5, \mu_3\mu_5) \nonumber \\
&\hspace{4.5cm} +(\mu_1-\mu_3+\mu_5, \mu_2+\mu_3-\mu_4-\mu_5)) \nonumber\\ 
&\cong \mathbb{Z}[x,y,z] / (x^2,y(x-z),y^2,z(y+z),xz),
\end{align}
where $x=\mu_3$, $y=\mu_4$ and $z=\mu_5$. 

\begin{lem}\label{lem:MBN_P5_1}
The maximal basis number of $\RPfive$ is $2$. 
\end{lem}

\begin{proof}
The proof is essentially the same as in Lemma~\ref{lem:MBN_P6_1}. 
It follows from \eqref{eq:HP5} that any element of $H^2(X_{P_5})$ is of the form $ax+by+cz$ with integers $a,b,c$ and an elementary computation shows that 
$$ (ax+by+cz)^2=(-2ab-2bc+c^2)z^2. $$
Therefore, s.v.e of $\RPsix$ are primitive elements in the set 
\begin{equation} \label{eq:sveP5}
\{ ax+by+cz \mid c^2=2b(a+c) \}.
\end{equation}
In particular, $x$ and $y$ are s.v.e. and since they form a part of $\ZZ$-basis, 
the maximal basis number of s.v.e. of $\RPfive$ is at least $2$. 

On the other hand, \eqref{eq:sveP5} shows that $c$ must be even. This implies that the basis maximal number must be at most 2, proving the lemma. 
\end{proof}

\begin{lem}\label{lem:tensor}
Let $A = \mathbb{Z}[x_1, \ldots, x_m] / \mathcal{I}_A$ and $B = \mathbb{Z}[y_1, \ldots, y_n] / \mathcal{I}_B$.
If $f$ is an s.v.e. of $A \otimes B$, then $f \in A$ or $f \in B$. In particular, the maximal basis number behaves additively with respect to tensor products. 
\end{lem}
\begin{proof}
Let $f = \sum_{i=1}^{m} a_i x_i + \sum_{j=1}^{n} b_j y_j$. Then, 
\begin{align*}
f^2
&= \left( \sum_{i=1}^{m} a_i x_i \right)^2 
+ \left( \sum_{j=1}^{n} b_j y_j \right)^2 
+ 2 \left( \sum_{i=1}^{m} a_i x_i \right) \left( \sum_{j=1}^{n} b_j y_j \right) \\
&= \left( \sum_{i=1}^{m} a_i x_i \right)^2 
+ \left( \sum_{j=1}^{n} b_j y_j \right)^2 
+ 2 \left( \sum_{i=1}^{m} \sum_{j=1}^{n} a_i b_j x_i y_j \right)
\end{align*}
Thus, we have $a_ib_j=0$ for any $i$ and $j$, i.e. 
we have either $a_i=0$ for any $i$ or $b_j=0$ for any $j$. 
\end{proof}

Under these preparations, we start to prove Theorem~\ref{thm:main} for the case (1). There are following five cases according to the number of vertices $V(P)$ and the parity of the dimension $d$ for smooth Fano $d$-polytopes $P$:
\begin{enumerate}
\item $V(P)=3d$ (in this case $d$ must be even),
\item $V(P)=3d-1$ and $d$ is odd,
\item $V(P)=3d-1$ and $d$ is even,
\item $V(P)=3d-2$ and $d$ is odd,
\item $V(P)=3d-2$ and $d$ is even.
\end{enumerate}
In (1) and (3) above, there is only one smooth Fano $d$-polytope by Theorems~\ref{thm:smoothFano3d} and~\ref{thm:smoothFano3d-1}, so it suffices to treat the remaining three cases. In the following, we may assume $d\ge 3$ since it is known (and easy to see) that smooth toric $d$-folds are distinguished by their cohomology rings when $d\le 2$. 

\subsection{The case where $V(P)=3d-1$ with $d$ odd $\ge 3$}

By Theorem~\ref{thm:smoothFano3d-1}, there are two $P$'s up to unimodular equivalence and their vertices are as follows:
\begin{align}
& e_1,\ \hspace{1em} -e_1, \hspace{1em}\ \pm e_2,\ \pm e_3,\ \pm (e_2-e_3), \ldots,\ \pm e_{d-1},\ \pm e_d,\ \pm (e_{d-1}-e_d), \tag{$Y_1^d$} \\
& e_1,\ -e_1+e_2,\ \pm e_2,\ \pm e_3,\ \pm (e_2-e_3), \ldots,\ \pm e_{d-1},\ \pm e_d,\ \pm (e_{d-1}-e_d). \tag{$Y_2^d$}
\end{align}
where the tags denote the corresponding toric Fano $d$-folds. In each case, the first two vertices correspond to the apices of the bipyramid and $\{ \pm e_{2k}, \pm e_{2k+1}, \pm (e_{2k}-e_{2k+1}) \}$
$(1 \leq k \leq \frac{d-1}{2})$ forms the hexagon $P_6$. 
One can see from the above that 
$$H^*(Y_1^d) = \mathbb{Z}[x] / (x^2) \otimes \RPsix^{\otimes \frac{d-1}{2}}, \quad 
H^*(Y_2^d) =H^*(Y_2^3)\otimes \RPsix^{\otimes \frac{d-3}{2}}$$ 
We number the eight vertices 
\[
e_2,\ \ e_3,\ \ -e_2+e_3,\ \ -e_2,\ \ -e_3,\ \ e_2-e_3,\ \ e_1,\ \ -e_1+e_2
\]
in $(Y_2^3)$ above from $1$ to $8$ and denote the corresponding elements in $H^2(Y_2^3)$ by $\mu_1,\dots,\mu_8$. Then we have 
\begin{align*}
& H^*(Y_2^3)\\
&\cong \mathbb{Z}[\mu_1,\mu_2,\mu_3,\mu_4,\mu_5,\mu_6,\mu_7,\mu_8] / ((\mu_1\mu_3, \mu_1\mu_4, \mu_1\mu_5, \mu_2\mu_4, \mu_2\mu_5, \mu_2\mu_6, \mu_3\mu_5, \mu_3\mu_6, \mu_4\mu_6, \mu_7\mu_8) \\
&\hspace{5.5cm} +(\mu_1-\mu_3-\mu_4+\mu_6+\mu_8, \mu_2+\mu_3-\mu_5-\mu_6, \mu_7-\mu_8)) \\
&\cong \mathbb{Z}[x,y,z,w,v] / (x(x+y-v),y(x+y-v),z(y-w-v), \\
&\hspace{5.5cm} y(x-z),z(z+w),w(z+w),xz,xw,yw,v^2),
\end{align*}
where $x=\mu_3$, $y=\mu_4$, $z=\mu_5$, $w=\mu_6$ and $v=\mu_8$. 
One can check that 
\begin{equation} \label{eq:sveY2}
\text{s.v.e. of $H^*(Y_2^3)=\{2y+2z-v, 2x+2y-v, z+w, v\}$\ (up to sign),} 
\end{equation}
and hence the maximal basis number of $H^*(Y_2^3)$ is $2$. Therefore, it follows from Lemma~\ref{lem:MBN_P6_1} and Lemma~\ref{lem:tensor} that the maximal basis numbers of $H^*(Y_1^d)$ and $H^*(Y_2^d)$ are as in the following table, so $H^*(Y_1^d)$ and $H^*(Y_2^d)$ are not isomorphic to each other. 

{\scriptsize
\begin{table}[H]
\centering
\begin{tabular}{|c|c|} \hline
Ring & Maximal basis number \\ \hline
$H^*(Y_1^d)$ & $1+3\cdot\frac{d-1}{2}$ \\ \hline
$H^*(Y_2^d)$ & $-1+3\cdot\frac{d-1}{2}$ \\ \hline
\end{tabular}
\end{table}
}

\subsection{The case where $V(P)=3d-2$ with $d$ odd $\ge 3$}

We first treat the case where $d=3$. By Theorem~\ref{thm:smoothFano3d-2}, the vertices of $P$ are one of the following: 
\begin{align}
& e_1,\ \hspace{1em} -e_1,\ \hspace{1em} e_2,\ \pm e_3,\ \pm (e_2-e_3), \tag{$Z_1$}\\
& e_1,\ -e_1+e_2,\ e_2,\ \pm e_3,\ \pm (e_2-e_3), \tag{$Z_2$} \\
& e_1,\ -e_1+e_3,\ e_2,\ \pm e_3,\ \pm (e_2-e_3), \tag{$Z_3$} \\
& e_1,\ -e_1-e_3,\ e_2,\ \pm e_3,\ \pm (e_2-e_3). \tag{$Z_4$}
\end{align}
In each case, $\{ e_2, \pm e_3, \pm (e_2-e_3) \}$ forms $P_5$ and the first two vertices correspond to the apices of the bipyramid over $P_5$. 

We claim that $Z_3$ and $Z_4$ are diffeomorphic. Indeed, the vertices in $Z_3$ are unimodularly equivalent to 
\[
-e_1,\ e_1+e_3,\ e_2,\ \pm e_3,\ \pm (e_2-e_3)
\]
through an automorphism $(x_1,x_2,x_3) \to (-x_1,x_2,x_3)$ of $\ZZ^3$, and these vectors agree with the vectors in $(Z_4)$ up to sign. Therefore, $Z_3$ and $Z_4$ are diffeomorphic by Lemma~\ref{diffeolemma}. 

We shall observe that the cohomology rings of $Z_1,Z_2,Z_3$ are not isomorphic to each other. 
Clearly $Z_1=\CC P^1\times X_{P_5}$ and hence 
\[
H^*(Z_1) \cong \mathbb{Z}[x] / (x^2) \otimes \RPfive. 
\]
To describe the cohomology rings of the remaining ones, we number the seven vertices 
\[
e_2,\ e_3,\ -e_2+e_3,\ -e_3,\ e_2-e_3,\ e_1,\ -e_1+* 
\]
from $1$ to $7$, where $*=e_2$ or $\pm e_3$. We denote the corresponding elements in $H^2(Z_i)$ for $i=2,3$ by $\mu_1,\dots,\mu_7$ and set 
\[
x=\mu_3,\ y=\mu_4,\ z=\mu_5,\ w=\mu_7.
\] 
Then we have 
\begin{align*}
H^*(Z_2) 
&\cong \mathbb{Z}[\mu_1,\mu_2,\mu_3,\mu_4,\mu_5,\mu_6,\mu_7] / ((\mu_1\mu_3, \mu_1\mu_4, \mu_2\mu_4, \mu_2\mu_5, \mu_3\mu_5, \mu_6\mu_7) \\
&\hspace{5cm} +(\mu_1-\mu_3+\mu_5+\mu_7, \mu_2+\mu_3-\mu_4-\mu_5,\mu_6-\mu_7)) \\
&\cong \mathbb{Z}[x,y,z,w] / (x(x-w),y(x-z-w),y(y-w),z(y+z),xz,w^2); \\
H^*(Z_3) 
&\cong \mathbb{Z}[\mu_1,\mu_2,\mu_3,\mu_4,\mu_5,\mu_6,\mu_7] / ((\mu_1\mu_3, \mu_1\mu_4, \mu_2\mu_4, \mu_2\mu_5, \mu_3\mu_5, \mu_6\mu_7) \\
&\hspace{5cm} +(\mu_1-\mu_3+\mu_5, \mu_2+\mu_3-\mu_4-\mu_5+\mu_7,\mu_6-\mu_7)) \\
&\cong \mathbb{Z}[x,y,z,w] / (x^2,y(x-z),y(y-w),z(y+z-w),xz,w^2).
\end{align*}
By an elementary computation using these presentations together with Lemmas~\ref{lem:MBN_P5_1} and~\ref{lem:tensor}, we obtain the following table; so $H^*(Z_i)$ for $i=1,2,3$ are not isomorphic to each other.

{\scriptsize
\begin{table}[H]
\centering
\begin{tabular}{|c|c|c|} \hline
Ring & s.v.e. & Maximal basis number \\ \hline
$H^*(Z_1)$ & infinitely many & $3$ \\ \hline
$H^*(Z_2)$ & $2x-w,2y-w,w$ & $1$ \\ \hline
$H^*(Z_3)$ & $x,2y-w,w$ & $2$ \\ \hline
\end{tabular}
\end{table}
}

Now we treat the case where $d \geq 5$. By Theorem~\ref{thm:smoothFano3d-2}, one can see that the vertices of $P$ are one of the following: 
\[
\text{the vertices in $(Z_i)$},\ \pm e_4, \pm e_5, \pm (e_4-e_5), \ldots, \pm e_{d-1}, \pm e_d, \pm (e_{d-1}-e_d) \tag{$Z^d_i$}
\]
for $i=1,\dots,4$ or 
\[
e_1, -e_1+e_4, e_2, \pm e_3, \pm (e_2-e_3), \pm e_4, \pm e_5, \pm (e_4-e_5), \ldots, \pm e_{d-1}, \pm e_d, \pm (e_{d-1}-e_d). \tag{$Z^d_5$}
\]
Note that $Z_5^d$ appears in the case $d \geq 5$ since $e_4$ and the vertices $e_2, \pm e_3, \pm (e_2-e_3)$ cannot be replaced each other by unimodular transformation. 
In each case, $\{\pm e_{2k}, \pm e_{2k+1}, \pm (e_{2k}-e_{2k+1})\}$ $(2 \leq k \leq \frac{d-1}{2})$ forms $P_6$ and the first two vertices correspond to the apices of the bipyramid. We also note that $\{e_2, \pm e_3, \pm (e_2-e_3)\}$ in $(Z_5^d)$ forms $P_5$. Therefore, one can see from the above that 
\begin{equation} \label{eq:Zid}
Z_i^d=Z_i\times (X_{P_6})^{\frac{d-3}{2}}\quad \text{for $i=1,2,3,4$}\qquad\text{and}\qquad Z_5^d=Y_2\times X_{P_5}\times (X_{P_6})^{\frac{d-5}{2}},
\end{equation}
where $Y_2$ denotes the Fano $3$-fold $Y_2^3$ in the previous subsection. Therefore, $H^*(Z_i^d)$ for $i=1,\ldots,5$ are the tensor product of the cohomology rings of the direct factors in $Z_i^d$. Since $Z_3$ and $Z_4$ are diffeomorphic as observed before, so are $Z_3^d$ and $Z_4^d$. It follows from Lemmas~\ref{lem:MBN_P6_1},~\ref{lem:MBN_P5_1} and~\ref{lem:tensor} that we obtain the following table. 

{\scriptsize
\begin{table}[H]
\centering
\begin{tabular}{|c|c|c|} \hline
Ring & Maximal basis number \\ \hline
$H^*(Z_1^d)$ & $3+3\cdot\frac{d-3}{2}$ \\ \hline
$H^*(Z_2^d)$ & $1+3\cdot\frac{d-3}{2}$ \\ \hline
$H^*(Z_3^d)$ & $2+3\cdot\frac{d-3}{2}$ \\ \hline
$H^*(Z_5^d)$ & $1+3\cdot\frac{d-3}{2}$ \\ \hline
\end{tabular}
\end{table}
}

Thus, we have to check $H^*(Z_2^d) \cong H^*(Z_5^d)$ or not. 
Suppose that there is an isomorphism 
$$F : H^*(Z_5^d)\cong H^*(Y_2)\otimes H^*(X_{P_5})\otimes H^*(X_{P_6})^{\otimes \frac{d-5}{2}} \rightarrow H^*(Z_2^d)\cong H^*(Z_2)\otimes H^*(X_{P_6})^{\otimes \frac{d-3}{2}}.$$
We note that any s.v.e. of $H^*(Z_5^d)$ belongs to one of the factors of the tensor products by Lemma~\ref{lem:tensor}. Let $f$ be an s.v.e. of $H^*(Z_5^d)$ which belongs to the factor $H^*(Y_2)$. We consider the following set 
\[
S(f):=\{ g\in H^2(Z_5^d)\mid \text{$g$ is an s.v.e. and $\{f, g\}$ is not a part of a $\ZZ$-basis of $H^2(Z_5^d)$}\}. 
\]
If $g\in S(f)$, then $g$ must belong to $H^*(Y_2)$ because otherwise $\{f,g\}$ is a part of a $\ZZ$-basis of $H^2(Z_5^d)$. Therefore $S(f)$ is a finite set by \eqref{eq:sveY2} and hence so is $S(F(f))$ because $F$ is an isomorphism. This together with Lemma~\ref{lem:MBN_P6_2} shows that $F(f)$ must be an s.v.e. of $H^*(Z_2)$. This means that $F$ sends the set of s.v.e. of $H^*(Y_2)$ to the set of s.v.e. of $H^*(Z_2)$. However, the cardinality of the former set up to sign is $4$ by \eqref{eq:sveY2} while that of the latter set up to sign is $3$ (see the previous subsection). This contradicts the injectivity of $F$. Hence, $H^*(Z_2^d)$ and $H^*(Z_5^d)$ are not isomorphic.

\subsection{The case where $V(P)=3d-2$ with $d$ even $\ge 4$} \label{subsec:3d-2even}

Theorem~\ref{thm:smoothFano3d-2} says that there are three types of $P$'s, i.e. (1), (2) and (3) in the theorem, but they have different face numbers (Remark~\ref{rem:facetnumber}). Therefore, the toric Fano $d$-folds in these different types can be distinguished by their cohomology rings. Since there are only one smooth Fano $d$-polytope in (2) and (3), its suffices to investigate the case (1). 

We first treat the case where $d=4$. One can see that the vertices of $P$ in Theorem.\ref{thm:smoothFano3d-2}(1) are one of the following: 
\begin{align}
& e_1,\ \hspace{1em} -e_1,\ \hspace{1em} e_2,\ \hspace{1em} -e_2,\ \hspace{1em} \pm e_3,\ \pm e_4,\ \pm (e_3-e_4), \tag{$W_1$}\\
& e_1,\ -e_1+e_2,\ e_2,\ \hspace{1em} -e_2,\ \hspace{1em} \pm e_3,\ \pm e_4, \ \pm (e_3-e_4), \tag{$W_2$} \\
& e_1,\ -e_1+e_2,\ e_2,\ -e_2+e_3,\ \pm e_3,\ \pm e_4,\ \pm (e_3-e_4), \tag{$W_3$} \\
& e_1,\ -e_1+e_3,\ e_2,\ \hspace{1em} -e_2,\ \hspace{1em} \pm e_3,\ \pm e_4,\ \pm (e_3-e_4), \tag{$W_4$} \\
& e_1,\ -e_1+e_3,\ e_2,\ -e_2+e_3,\ \pm e_3,\ \pm e_4,\ \pm (e_3-e_4), \tag{$W_5$} \\
& e_1,\ -e_1+e_3,\ e_2,\ -e_2+e_4,\ \pm e_3,\ \pm e_4,\ \pm (e_3-e_4), \tag{$W_6$} \\
& e_1,\ -e_1+e_3,\ e_2,\ -e_2-e_3,\ \pm e_3,\ \pm e_4,\ \pm (e_3-e_4), \tag{$W_7$} \\
& e_1,\ -e_1+e_3,\ e_2,\ -e_2-e_4,\ \pm e_3,\ \pm e_4,\ \pm (e_3-e_4). \tag{$W_8$} 
\end{align}
In each case, $\{ \pm e_3, \pm e_4, \pm (e_3-e_4) \}$ forms the hexagon $P_6$ and the first two vertices and second two vertices correspond to the apices of the double bipyramid over $P_6$. 
Note that those $W_1,\ldots, W_8$ can be obtained by considering the pair of the first two vertices and the second two vertices, i.e. 
where each of two segments coming from apices intersect. We can check that those are exactly the possible cases. 

We claim that $W_5$ is diffeomorphic to $W_7$. Indeed, the vertices in $W_5$ are unimodularly equivalent to 
\[
e_1,\ -e_1+e_3,\ -e_2,\ e_2+e_3,\ \pm e_3,\ \pm e_4,\ \pm (e_3-e_4)
\]
through an automorphism $(x_1,x_2,x_3,x_4) \to (x_1,-x_2,x_3,x_4)$ of $\ZZ^4$, and these vectors agree with the vectors in $W_7$ up to sign. Therefore, $W_5$ and $W_7$ are diffeomorphic by Lemma~\ref{diffeolemma}. The same argument shows that $W_6$ is diffeomorphic to $W_8$. 

We shall observe that $H^*(W_i)$ for $i=1,\dots,6$ are not isomorphic to each other. One can easily see 
\begin{align*}
H^*(W_1) \cong \mathbb{Z}[x,y] / (x^2,y^2) \otimes \RPsix,\;\;\text{and}\;\; H^*(W_2) \cong \mathbb{Z}[x,y] / (x^2,y(y-x)) \otimes \RPsix. 
\end{align*}
To describe the cohomology rings of the remaining ones, we number the ten vertices 
\[
e_3,\ e_4,\ -e_3+e_4,\ -e_3,\ -e_4,\ e_3-e_4,\ e_1,\ -e_1+*,\ e_2,\ -e_2+\star
\]
from $1$ to $10$, where $*=e_2$ or $e_3$ and $\star=0, \pm e_3$ or $\pm e_4$. We denote the corresponding elements in $H^2(W_i)$ for $i=3,4,\dots,8$ by $\mu_1,\dots,\mu_{10}$ and set 
\[
x=\mu_3,\ y=\mu_4,\ z=\mu_5,\ w=\mu_6,\ v=\mu_8,\ u=\mu_{10}.
\] 
Then we have
\begin{align*}
\hspace{-0em}H^*(W_3) 
&\hspace{-0em}\cong \mathbb{Z}[\mu_1,\mu_2,\mu_3,\mu_4,\mu_5,\mu_6,\mu_7,\mu_8,\mu_9,\mu_{10}] / \\
&\hspace{3em}((\mu_1\mu_3, \mu_1\mu_4, \mu_1\mu_5, \mu_2\mu_4, \mu_2\mu_5, \mu_2\mu_6, \mu_3\mu_5, \mu_3\mu_6, \mu_4\mu_6, \mu_7\mu_8, \mu_9\mu_{10}) \\
&\hspace{3.5em} +(\mu_1-\mu_3-\mu_4+\mu_6+\mu_{10}, \mu_2+\mu_3-\mu_5-\mu_6, \mu_7-\mu_8, \mu_8+\mu_9-\mu_{10})) \\
&\hspace{-0em}\cong \mathbb{Z}[x,y,z,w,v,u] / (x(x+y-u),y(x+y-u),z(y-w-u), \\
&\hspace{8em} y(x-z),z(z+w),w(z+w),xz,xw,yw,v^2,u(u+v)).
\end{align*}
\begin{align*}
\hspace{-0em}H^*(W_4) 
&\hspace{-0em}\cong \mathbb{Z}[\mu_1,\mu_2,\mu_3,\mu_4,\mu_5,\mu_6,\mu_7,\mu_8,\mu_9,\mu_{10}] / \\
&\hspace{4em}((\mu_1\mu_3, \mu_1\mu_4, \mu_1\mu_5, \mu_2\mu_4, \mu_2\mu_5, \mu_2\mu_6, \mu_3\mu_5, \mu_3\mu_6, \mu_4\mu_6, \mu_7\mu_8, \mu_9\mu_{10}) \\
&\hspace{5.5em} +(\mu_1-\mu_3-\mu_4+\mu_6+\mu_8, \mu_2+\mu_3-\mu_5-\mu_6, \mu_7-\mu_8, \mu_9-\mu_{10})) \\
&\hspace{-0em}\cong \mathbb{Z}[x,y,z,w,v,u] / (x(x+y-v),y(x+y-v),z(y-w-v), \\
&\hspace{8em} y(x-z),z(z+w),w(z+w),xz,xw,yw,v^2,u^2).
\end{align*}
\begin{align*}
\hspace{-0em}H^*(W_5)
&\hspace{-0em}\cong \mathbb{Z}[\mu_1,\mu_2,\mu_3,\mu_4,\mu_5,\mu_6,\mu_7,\mu_8,\mu_9,\mu_{10}] / \\
&\hspace{4em}((\mu_1\mu_3, \mu_1\mu_4, \mu_1\mu_5, \mu_2\mu_4, \mu_2\mu_5, \mu_2\mu_6, \mu_3\mu_5, \mu_3\mu_6, \mu_4\mu_6, \mu_7\mu_8, \mu_9\mu_{10}) \\
&\hspace{4.5em} +(\mu_1-\mu_3-\mu_4+\mu_6+\mu_8+\mu_{10}, \mu_2+\mu_3-\mu_5-\mu_6, \mu_7-\mu_8, \mu_9-\mu_{10})) \\
&\hspace{-0em}\cong \mathbb{Z}[x,y,z,w,v,u] / (x(x+y-v-u),y(x+y-v-u),z(y-w-v-u), \\
&\hspace{8em} y(x-z),z(z+w),w(z+w),xz,xw,yw,v^2,u^2).
\end{align*}
\begin{align*}
\hspace{-0em}H^*(W_6) 
&\hspace{-0em}\cong \mathbb{Z}[\mu_1,\mu_2,\mu_3,\mu_4,\mu_5,\mu_6,\mu_7,\mu_8,\mu_9,\mu_{10}] / \\
&\hspace{4em}((\mu_1\mu_3, \mu_1\mu_4, \mu_1\mu_5, \mu_2\mu_4, \mu_2\mu_5, \mu_2\mu_6, \mu_3\mu_5, \mu_3\mu_6, \mu_4\mu_6, \mu_7\mu_8, \mu_9\mu_{10}) \\
&\hspace{4.5em} +(\mu_1-\mu_3-\mu_4+\mu_6+\mu_8, \mu_2+\mu_3-\mu_5-\mu_6+\mu_{10}, \mu_7-\mu_8, \mu_9-\mu_{10})) \\
&\hspace{-0em}\cong \mathbb{Z}[x,y,z,w,v,u] / (x(x+y-v),y(x+y-v),z(y-w-v), \\
&\hspace{8em} y(x-z+u),z(z+w-u),w(z+w-u),xz,xw,yw,v^2,u^2).
\end{align*}
By an elementary computation using the above presentations together with Lemma~\ref{lem:MBN_P6_1}, we obtain the following table; so $H^*(W_i)$ for $i=1,2,3$ are not isomorphic to each other. 

{\scriptsize
\begin{table}[H]
\centering
\begin{tabular}{|c|c|c|} \hline
Ring & s.v.e. & Maximal basis number \\ \hline
$H^*(W_1)$ & infinitely many & $5$ \\ \hline
$H^*(W_2)$ & infinitely many & $4$ \\ \hline
$H^*(W_3)$ & $z+w,v+2u,v$ & $2$ \\ \hline
$H^*(W_4)$ & $2y+2z-v,2x+2y-v,v,z+w,u$ & $3$ \\ \hline
$H^*(W_5)$ & $z+w,v,u$ & $3$ \\ \hline
$H^*(W_6)$ & $2x+2y-v,v,2z+2w-u,u$ & $2$ \\ \hline
\end{tabular}
\medskip
\label{table:sveWi}
\end{table}
}

Now we treat the case where $d\ge 6$. The vertices of $P$ in Theorem~\ref{thm:smoothFano3d-2} (1) are either 
\[
\text{the vertices in $(W_i)$},\ \pm e_5,\ \pm e_6,\ \pm (e_5-e_6), \ldots, \pm e_{d-1},\ \pm e_d,\ \pm (e_{d-1}-e_d)\tag{$W^d_i$}
\]
for $i=1,\dots,8$ or 
\[
e_1,\ -e_1+e_3,\ e_2,\ -e_2+e_5,\ \pm e_3, \pm e_4, \pm (e_3-e_4), \ldots, \pm e_{d-1}, \pm e_d, \pm (e_{d-1}-e_d). \tag{$W^d_9$}
\]
Note that $W_9^d$ appears in the case $d \geq 6$ since $e_5$ and the vertices $\pm e_3, \pm e_4, \pm (e_3-e_4)$ cannot be replaced each other by unimodular transformation. 
The first two vertices and second two vertices correspond to apices of the double bipyramid and $\{ \pm e_{2k-1}, \pm e_{2k}, \pm (e_{2k-1}-e_{2k}) \}$ $(3 \leq k \leq \frac{d}{2})$ forms $P_6$. 
One can see from the above that 
\begin{equation} \label{eq:Wid}
W_i^d=W_i\times (X_{P_6})^{\frac{d-4}{2}} \quad\text{for $i=1,\dots,8$}\quad \text{and}\quad
W_9^d=Y_2\times Y_2\times (X_{P_6})^{\frac{d-6}{2}},
\end{equation}
where $Y_2$ is the $3$-fold $Y_2^3$ in subsection 1.2. 
Since $W_5$ (resp. $W_6$) is diffeomorphic to $W_7$ (resp. $W_8$), $W_5^d$ (resp. $W_6^d$) is diffeomorphic to $W_7^d$ (resp. $W_8^d$). 
The maximal basis number of $H^*(X_{P_6})$ is $3$ by Lemma~\ref{lem:MBN_P6_1} and that of $H^*(Y_2)$ is $2$ by \eqref{eq:sveY2}. Therefore, it follows from \eqref{eq:Wid}, Table~\ref{table:sveWi}, and Lemma~\ref{lem:tensor} that we obtain the following table.

{\scriptsize
\begin{table}[H]
\centering
\begin{tabular}{|c|c|c|} \hline
Ring & Maximal basis number \\ \hline
$H^*(W^d_1)$ & $5+3\cdot\frac{d-4}{2}$ \\ \hline
$H^*(W^d_2)$ & $4+3\cdot\frac{d-4}{2}$ \\ \hline
$H^*(W^d_3)$ & $2+3\cdot\frac{d-4}{2}$ \\ \hline
$H^*(W^d_4)$ & $3+3\cdot\frac{d-4}{2}$ \\ \hline
$H^*(W^d_5)$ & $3+3\cdot\frac{d-4}{2}$ \\ \hline
$H^*(W^d_6)$ & $2+3\cdot\frac{d-4}{2}$ \\ \hline
$H^*(W^d_9)$ & $-2+3\cdot\frac{d-4}{2}$ \\ \hline
\end{tabular}
\end{table}
}

Thus, we have to check whether $H^*(W^d_3) \cong H^*(W^d_6)$ or not and $H^*(W^d_4) \cong H^*(W^d_5)$ or not, but the same argument as in the last part of the previous subsection shows that they are not isomorphic. Indeed, if there is an isomorphism $F : H^*(W^d_6) \rightarrow H^*(W^d_3)$, then $F$ must send the set of s.v.e. of $H^*(W_6)$ to that of $H^*(W_3)$. However, the cardinality of the former set up to sign is $4$ while that of the latter set is $3$ (see Table~\ref{table:sveWi}). This contradicts the injectivity of $F$. Therefore, $H^*(W^d_3)$ is not isomorphic to $H^*(W^d_6)$. The same argument shows that $H^*(W^d_4)$ is not isomorphic to $H^*(W^d_5)$.

\section{The case of dimension $3$}\label{sec:dim3}

In the remaining sections, we use the database by {\O}bro mentioned in Introduction.  Each smooth Fano polytope or toric Fano manifold has ID.  We will denote the toric Fano manifold with ID number $q$ by $X_q$. 

There are 18 variety-isomorphism classes of toric Fano $3$-folds, in other words, 18 unimodular equivalence classes of smooth Fano $3$-polytopes. In this section, we will classify them up to diffeomorphism. It turns out that the cohomological rigidity holds for them.  More precisely, there are 16 diffeomorphism classes as is shown in Table~\ref{table:diff3fold} below, where ID numbers whose toric Fano $3$-folds are diffeomorphic are enclosed by curly braces, and five ID numbers before $\|$ are Bott manifolds. In Table~\ref{table:diff3fold}, $V(P)$ is the number of vertices of $P$, $P$ shows the unimodular equivalence classes of smooth Fano $3$-polytopes, $H^*$ shows the isomorphism classes of integer cohomology rings, and Diff shows the diffeomorphism classes. 

{\scriptsize
\begin{table}[H]
\centering
\begin{tabular}{|c|c|c|c|c|} \hline
$V(P)$ & $P$ & $H^*$ & Diff & ID \\ \hline\hline
4 & 1 & 1 & 1 & 23 \\ \hline 
5 & 4 & 4 & 4 & 7, 19, 20, 22 \\ \hline
6 & 7 & 6 & 6 & $\{11,18\}$, 12, 17, 21 $\|$ 6, 16\\ \hline
7 & 4 & 3 & 3 & 8, $\{10, 13\}$, 14 \\ \hline
8 & 2 & 2 & 2 & 9, 15\\ \hline
{total} & 18 & 16 & 16 & \\ \hline
\end{tabular} 
\medskip
\caption{Diffeomorphism classification of toric Fano $3$-folds} 
\label{table:diff3fold}
\end{table}
}

We shall explain how we obtain Table~\ref{table:diff3fold}. 
There is only one smooth Fano $3$-polytope $P$ with $V(P)=4$, so there is nothing to prove in this case. The case where $V(P)=7$ or $8$ is treated in Section~\ref{sec:Picard}. Indeed, toric Fano $3$-folds $X_q$ with $q=8, 10, 13, 14, 9, 15$ are respectively $Z_2, Z_3, Z_4, Z_1, Y_2^3, Y_1^3$ in Section~\ref{sec:Picard}. Therefore, it suffices to investigate the case where $V(P)=5$ or $6$. 

\medskip
\noindent
{\bf Convention.} 
\begin{enumerate}
\item The vertices of a smooth Fano $3$-polytope $P$ are shown in the database of {\O}bro and we number them as $1,2,\dots$ in the order shown in the database.  
\item The first three vertices of $P$ are the standard basis of $\ZZ^3$, so we omit them and write the vertices from $4$th in Tables~\ref{tab:VMNF5} and~\ref{tab:VMNF6} below.  
\item Minimal nonfaces of $P$ are described using the numbering of the vertices of $P$.  
\item $\mathcal{I}$ denotes the ideal of the cohomology ring $H^*(X_q)$ and its minimal generators are described in the tables. 
\item s.v.e. and c.v.e. in the tables are up to sign unless the coefficient is $\ZZ/2$. 
\end{enumerate}

\subsection{The case where $V(P) = 5$}

In this case, there are four smooth Fano $3$-polytopes as shown in Table~\ref{tab:VMNF5}.  They are all combinatorially equivalent to a direct sum of a $2$-simplex and a $1$-simplex, so the corresponding toric Fano $3$-folds are generalized Bott manifolds.  

{\scriptsize
\begin{table}[H]
\centering
\begin{tabular}{|c|c|c|} \hline
ID & vertices of $P$ from $4$th & minimal nonfaces \\ \hline
7 & $(-1,-1,2), (0,0,-1)$ & $35, 124$ \\ \hline
19 & $(-1,0,1), (0,-1,-1)$ & $14, 235$ \\ \hline
20 & $(-1,-1,1), (0,0,-1)$ & $35, 124$ \\ \hline
22 & $(-1,0,0), (0,-1,-1)$ & $14, 235$ \\ \hline
\end{tabular}
\medskip
\caption{Vertices and minimal nonfaces of $P$ with $V(P)=5$} \label{tab:VMNF5}
\end{table}
}

We denote the degree two cohomology element corresponding to $4$th and $5$th vertices by $x$ and $y$, respectively. Then the cohomology ring of each toric Fano $3$-fold with ID number in Table~\ref{tab:VMNF5} is the quotient of a polynomial ring $\ZZ[x,y]$ by an ideal $\mathcal{I}$. By an elementary computation, we obtain Table~\ref{d=3v=5} which shows that those four cohomology rings are not isomorphic to each other.

{\scriptsize
\begin{table}[H]
\centering
\begin{tabular}{|c|c|c|c|c|} \hline
ID & $\mathcal{I}$ & s.v.e. & s.v.e. $\ZZ/2$ & c.v.e. $\ZZ/3$ \\ \hline
7 & $x^3,y(2x-y)$ & $\emptyset$ & $(y)$ &  \\ \hline
{19} & $x^2, y^2(x-y)$ & $x$ & $(x)$ & $(x)$ \\ \hline
20 & $x^3,y(x-y)$ & $\emptyset$ & $\emptyset$ &  \\ \hline
{22} & $x^2,y^3$ & $x$ & $(x)$ & $(x,y)$ \\ \hline
\end{tabular}
\medskip
\caption{Ideals and invariants when $V(P)=5$}\label{d=3v=5}
\end{table}
}


\subsection{The case where $V(P)=6$}

In this case, there are seven smooth Fano $3$-polytopes $P$ as shown in Table~\ref{tab:VMNF6}.  The polytopes with ID numbers 11, 12, 18, 18, 21 are combinatorially equivalent to a cross-polytope, so the corresponding toric Fano $3$-folds are Bott manifolds.  

{\scriptsize
\begin{table}[H]
\centering
\begin{tabular}{|c|c|c|} \hline
ID & vertices of $P$ from $4$th & minimal nonfaces \\ \hline
11 & $(-1,0,1), (0,-1,1), (0,0,-1)$ & $14,25,36$ \\ \hline
12 & $(-1,0,1), (0,-1,0), (0,1,-1)$ & $14,25,36$ \\ \hline
17 & $(-1,0,1), (0,-1,0), (0,0,-1)$ & $14,25,36$ \\ \hline
18 & $(-1,0,1),(0,-1,-1),  (0,0,-1)$ & $14,25,36$ \\ \hline
21 & $(-1,0,0), (0,-1,0), (0,0,-1)$ & $14,25,36$ \\ \hline\hline
6& $(-1,-1,2), (0,1,-1), (0,0,-1)$ & $26,35,36,124,145$ \\ \hline
16& $(-1,0,1), (1,0,-1), (-1,-1,0)$ & $14,35,45,126,236$ \\ \hline
\end{tabular}
\medskip
\caption{Vertices and minimal nonfaces of $P$ with $V(P)=6$} \label{tab:VMNF6}
\end{table}
}

\begin{rem}
We interchanged $5$th and $6$th vertices in \cite{Oeb} for ID numbers 12 and 18 so that the minimal nonfaces have the same numbering as others. 
\end{rem}
Through an automorphism $(x_1,x_2,x_3)\to (x_1,-x_2,x_3)$ of $\ZZ^3$, the vertices of ID number 11 are unimodularly equivalent to 
$(1,0,0),\ (0,-1,0),\ (0,0,1),\ (-1,0,1),\ (0,1,1),\ (0,0,-1)$
and these agree with the vertices of ID number 18 up to sign, so $X_{11}$ is diffeomorphic to $X_{18}$ by Lemma~\ref{diffeolemma}. 

We denote the degree two cohomology element corresponding to $4$th, $5$th and $6$th vertices by $x$, $y$ and $z$ respectively. Then the cohomology ring of each toric Fano $3$-fold with ID number in Table~\ref{tab:VMNF6} is the quotient of a polynomial ring $\ZZ[x,y,z]$ by an ideal $\mathcal{I}$. By an elementary computation, we obtain Table~\ref{d=3v=6}, which shows that the cohomology rings in the tables are not isomorphic to each other: 

{\scriptsize
\begin{table}[H]
\centering
\begin{tabular}{|c|c|c|c|} \hline
ID & $\mathcal{I}$ & s.v.e. & maximal basis number \\ \hline
11(18) & $x^2,y^2,z(x+y-z)$ & $x,y$ & $2$ \\ \hline
{12} & $x^2,y(y-z),z(x-z)$ & $x,x-2z$ & $1$ \\ \hline
17 & $x^2,y^2,z(x-z)$ & $x,y,x-2z$ & $2$ \\ \hline
21 & $x^2,y^2,z^2$ & $x,y,z$ & $3$ \\ \hline\hline
6 & $z(x-y),y(2x-y-z),z(x-z),x^3,x^2y$ & $\emptyset$ & $0$ \\ \hline
16 & $x(x+z),y^2,xy,z^3,z^2(x-y)$ & $y$ & $1$ \\ \hline
\end{tabular}
\medskip
\caption{Ideals and invariants when $V(P)=6$}\label{d=3v=6}
\end{table}
}


\section{The case of dimension $4$}\label{sec:dim4} 

There are $124$ variety-isomorphism classes of toric Fano $4$-folds, in other words, $124$ unimodular equivalence classes of smooth Fano polytopes $P$ of dimension $4$. In this section we will classify them up to diffeomorphism. It turns out that the cohomological rigidity holds for them except for $X_{50}$ and $X_{57}$. The $X_{50}$ and $X_{57}$ have isomorphic cohomology rings and their total Pontryagin classes are preserved under an isomorphism between their cohomology rings but we do not know whether they are diffeomorphic or not. 

In Table~\ref{table:diff4fold} below, ID numbers whose toric Fano $4$-folds are diffeomorphic are enclosed by curly braces as before.  Thirteen toric Fano $4$-folds with the ID numbers in the upper two lines in the row of $(V(P),F(P))=(8,16)$ are Bott manifolds and toric Fano $4$-folds for $(V(P),F(P))=(6,8), (6,9), (7,12)$ are generalized Bott manifolds, where $V(P)$ denotes the number of vertices of $P$ as before and $F(P)$ denotes the number of facets of $P$.  

{\scriptsize 
\begin{table}[H]
\centering
\begin{tabular}{|c|c|c|c|c|c|} \hline
$V(P)$& $F(P)$ & $P$ & $H^*$ & Diff & ID \\ \hline 
5 & 5 & 1 & 1 & 1 & 147\\ \hline 
6 & 8 & 5 & 5 & 5 & 25, 138, 139, 144, 145 \\ \hline
6 & 9 & 4 & 3 & 3 & 44, $\{70, 141\}$, 146 \\ \hline 
7 & 11 & 3 & 3 & 3 & 24, 127, 128 \\ \hline
7 & 12 & 19 & 16 & 16 & $\{30, 43\}$, 31, 35, 42, 49, 66, $\{68, 134\}$, 109 \\
& & & & & 117, $\{129, 136\}$, 132, 133, 135, 140, 143, 197 \\ \hline
7 & 13 & 6 & 6 & 6 & 40, 41, 60, 64, 69, 137 \\ \hline 
8 & 15 & 10 & 7 & 7 & 26, $\{28,32\}$, 45, 48, $\{67, 118\}$, $\{123, 125\}$, 124 \\ \hline
8 & 16 & 28 & 23 & 23 & $\{74, 96\}$, 75, $\{83, 108\}$, $\{95, 131\}$\\
& & & & & 105, 106, 112, 114, 130, 142 \\ 
& & & & & $\{29, 39\}$, 33, 34, 37, 38, 47, 59\\
& & & & & 93, 94, 104, $\{111, 116\}$, 115, 126 \\ \hline
8 & 17 & 7 & 6 & 6 or 7 & 36, (50, 57), 58, 61, 65, 110 \\ \hline
8 & 18 & 2 & 2 & 2 & 53, 55 \\ \hline 
9 & 18 & 4 & 4 & 4 & 27, 46, 119, 122 \\ \hline
9 & 20 & 17 & 10 & 10 & 71, $\{73, 76, 92\}$, $\{77, 88\}$, 79, $\{81, 103\}$ \\
& & & & & $\{82, 91, 107\}$, 84, $\{90, 113\}$, 102, 120 \\ \hline
9 & 21 & 4 & 4 & 4 & 51, 52, 56, 89 \\ \hline
9 & 23 & 1 & 1 & 1 & 62 \\ \hline
9 & 24 & 1 & 1 & 1 & 54 \\ \hline
10 & 24 & 8 & 6 & 6 & $\{ 72, 87\}$, $\{78, 86\}$, 80, 85, 101, 121 \\ \hline
10 & 25 & 1 & 1 & 1 & 98 \\ \hline
10 & 30 & 1 & 1 & 1 & 63 \\ \hline
11 & 30 & 1 & 1 & 1 & 99 \\ \hline
12 & 36 & 1 & 1 & 1 & 100 \\ \hline
\multicolumn{2}{|c|}{total} & 124 & 102 & 102 or 103 & \\ \hline
\end{tabular} 
\medskip
\caption{Diffeomorphism classification of toric Fano $4$-folds} \label{table:diff4fold}
\end{table}
}

Our approach to obtain the table above is the same as the case of dimension $3$ but the analysis of the cohomology rings in dimension $4$ becomes much more complicated and there are many more cases to investigate. 

When $V(P)=5, 11, 12$ or $(V(P),F(P))=(9,23), (9,24), (10,25), (10,30)$, there is only one smooth Fano $4$-polytope; so there is nothing to prove in these cases. Moreover, the case $(V(P),F(P))=(10,24)$ is treated in Section~\ref{sec:Picard}. Indeed, toric Fano $4$-folds with ID numbers 72, 78, 80, 85, 86, 87, 101, 121 are respectively $W_5, W_6, W_3, W_4, W_8, W_7, W_2, W_1$ in Subsection~\ref{subsec:3d-2even}. Therefore, it suffices to investigate the remaining cases. We shall carry out this task one by one in this section. 

\medskip
\noindent
{\bf Convention}. 
\begin{enumerate}
\item The vertices of a smooth Fano $4$-polytope $P$ are shown in the database by {\O}bro and we number them as 1, 2, ... in the order shown in the database. 
\item The first four vertices of $P$ are the standard basis of $\ZZ^4$, so we omit them and write the vertices from 5th in the tables below. 
\item Minimal nonfaces of $P$ are described using the numbering of the vertices of $P$.  
\item $\mathcal{I}$ denotes the ideal of the cohomology ring $H^*(X_q)$ and its minimal generators are described in the tables. 
\item s.v.e., c.v.e. and $4$-v.e. in the tables are up to sign unless the coefficient is $\ZZ/2$. 
\item $X_p\cong X_q$ means that $X_p$ is diffeomorphic to $X_q$. 
\item $H^*(X_p)\ncong H^*(X_q)$ means that the cohomology rings are not isomorphic (as graded rings). 
\item The degree two cohomology elements corresponding to $5$th, $6$th, $7$th, $8$th, $9$th vertices are respectively denoted by $x, y, z, u, v$. 
\end{enumerate}

\subsection{The case where $V(P)=6$}

We take two cases according to the values of $F(P)$. 

\subsubsection{$(V(P),F(P))=(6,8)$} 
In this case, there are five smooth Fano $4$-polytopes and they are all combinatorially equivalent to a direct sum of a $3$-simplex and a $1$-simplex, so the corresponding toric Fano $4$-folds are generalized Bott manifolds. Using the data in Table~\ref{tab:VMNF68}, we obtain Table~\ref{d=4v=6f=8} which shows that the five cohomology rings are not isomorphic to each other. 

{\scriptsize
\begin{table}[H]
\centering
\begin{tabular}{|c|c|c|} \hline
ID & vertices of $P$ from $5$th & minimal nonfaces \\ \hline
25 & $(-1,-1,-1,3), (0,0,0,-1))$ & 1235, 46\\ \hline
138 & $(-1,0,0,1), (0,-1,-1,-1)$ & 2346, 15\\ \hline
139 & $(-1,-1,-1,2), (0,0,0,-1)$ &1235, 46 \\ \hline
144 & $(-1,0,0,0), (1,-1,-1,-1)$ & 2346, 15\\ \hline
145 & $(-1,0,0,0), (0,-1,-1,-1)$ & 2346, 15\\ \hline
\end{tabular}
\medskip
\caption{Vertices and minimal nonfaces of $P$ with $(V(P),F(P))=(6,8)$} 
\label{tab:VMNF68}
\end{table}
\begin{table}[h]
\centering
\begin{tabular}{|c|c|c|c|c|c|} \hline
ID & $\mathcal{I}$ & s.v.e. & s.v.e. $\ZZ/2$ & 4-v.e. $\ZZ/2$ & c.v.e. $\ZZ/3$ \\ \hline
25 & $x^4,y(y-3x)$ & $\emptyset$ & $\emptyset$ & $x$ & $y$ \\ \hline
138 & $x^4,y(y-x)$ & $\emptyset$ & $\emptyset$ & $x$ & $\emptyset$ \\ \hline
139 & $x^4,y(y-2x)$ & $\emptyset$ & $y$ &  &  \\ \hline
144 & $x^3(x-y),y^2$ & $y$ & $y$ & $y$ &  \\ \hline
145 & $x^4,y^2$ & $y$ & $y$ & $(x,y)$ &  \\ \hline
\end{tabular}
\medskip
\caption{Ideals and invariants when $(V(P), F(P))=(6,8)$}
\label{d=4v=6f=8}
\end{table}
}


\subsubsection{$(V(P),F(P))=(6,9)$}

In this case, there are four smooth Fano $4$-polytopes and they are all combinatorially equivalent to a direct sum of two $2$-simplices, so the corresponding toric Fano $4$-folds are generalized Bott manifolds.

{\scriptsize
\begin{table}[H]
\centering
\begin{tabular}{|c|c|c|} \hline
ID & vertices of $P$ from $5$th & minimal nonfaces \\ \hline
44 & $(-1,-1,0,2), (0,0,-1,-1)$ & 125, 346\\ \hline
70 & $(-1,-1,1,1), (0,0,-1,-1)$ & 125, 346\\ \hline
141 & $(-1,-1,0,1), (0,0,-1,-1)$ & 125, 346\\ \hline
146 & $(-1,-1,0,0), (0,0,-1,-1)$ & 125, 346\\ \hline
\end{tabular}
\medskip
\caption{Vertices and minimal nonfaces of $P$ with $(V(P),F(P))=(6,9)$} \label{tab:VMNF69}
\end{table}
}

We see that 
\begin{equation} \label{eq:diff69}
X_{70}\cong X_{141},
\end{equation}
see Subsection~\ref{subsec:diffeo}. On the other hand, using the data in Table~\ref{tab:VMNF69}, we obtain Table~\ref{d=4v=6f=9} which shows that the three cohomology rings are not isomorphic to each other.

{\scriptsize
\begin{table}[H]
\centering
\begin{tabular}{|c|c|c|c|} \hline
ID & $\mathcal{I}$ & c.v.e. & c.v.e. $\ZZ/2$ \\ \hline
44 & $x^3,y^2(2x-y)$ & $x$ & $x,y$ \\ \hline
70(141) & $x^3,y(x-y)^2$ & $x$ & $x$ \\ \hline
146 & $x^3,y^3$ & $x,y$ &  \\ \hline
\end{tabular}
\medskip
\caption{Ideals and invariants when $(V(P),F(P))=(6,9)$}
\label{d=4v=6f=9}
\end{table}
}


\subsection{The case where $V(P)=7$}

\subsubsection{$(V(P),F(P))=(7,11)$}
In this case, there are three smooth Fano $4$-polytopes and they are all combinatorially equivalent. 
Using the data in Table~\ref{tab:VMNF711}, we obtain Table~\ref{d=4v=7f=11} which shows that the three cohomology rings are not isomorphic to each other. 

{\scriptsize
\begin{table}[H]
\centering
\begin{tabular}{|c|c|c|} \hline
ID & vertices of $P$ from $5$th & minimal nonfaces \\ \hline
24 & $(-1,-1,-1,3), (0,0,1,-1), (0,0,0,-1)$ & 1235, 1256, 37, 46, 47\\ \hline
127 & $(-1,0,0,1), (1,0,0,-1), (-1,-1,-1,0)$ & 1237, 2347, 15, 46, 56\\ \hline
128 & $(-1,0,0,1), (-1,0,0,0), (2,-1,-1,-1)$ & 2347, 2357, 15, 16, 46\\ \hline
\end{tabular}
\medskip
\caption{Vertices and minimal nonfaces of $P$ with $(V(P),F(P))=(7,11)$} \label{tab:VMNF711}
\end{table}
\begin{table}[H]
\centering
\begin{tabular}{|c|c|c|c|} \hline
ID & $\mathcal{I}$ & s.v.e. & c.v.e. $\ZZ/3$ \\ \hline
24 & $z(x-y),y(y+z-3x),z(z-2x),x^4,x^3y$ & $\emptyset$ & $y+z$ \\ \hline
127 & $x(x+z),y^2,xy,z^4,z^3(y-x)$ & $y$ &  \\ \hline
128 & $x(x+y-2z),y(y-z),y(z-x),z^4,xz^3$ & $\emptyset$ & $\emptyset$ \\ \hline
\end{tabular}
\medskip
\caption{Ideals and invariants when $(V(P),F(P))=(7,11)$}
\label{d=4v=7f=11}
\end{table}
}


\subsubsection{$(V(P),F(P))=(7,12)$}

In this case, there are 19 smooth Fano $4$-polytopes and they are all combinatorially equivalent to a direct sum of $2$-simplex and two $1$-simplices, so the corresponding toric Fano $4$-folds are generalized Bott manifolds.

{\scriptsize
\begin{table}[H]
\begin{tabular}{|c|c|c|} \hline
ID & vertices of $P$ from $5$th & minimal nonfaces \\ \hline
30 & $ (-1,-1,0,2), (0,0,-1,1), (0,0,0,-1)$ & 125, 36, 47\\ \hline
31 & $(-1,-1,0,2), (0,0,1,-1), (0,0,-1,0)$ & 125, 37, 46\\ \hline
35 & $(-1,-1,0,2), (0,1,-1,0), (0,0,0,-1)$ & 125, 36, 47\\ \hline
42& $(-1,-1,0,2), (0,0,-1,0), (0,0,0,-1)$ & 125, 36, 47\\ \hline
43& $(-1,-1,0,2), (0,0,0,-1), (0,0,-1,-1)$ &125, 37, 46 \\ \hline
49& $(-1,-1,1,1), (0,0,-1,1), (0,0,0,-1)$ & 125, 36, 47\\ \hline
66& $(-1,-1,1,1), (0,0,-1,0), (0,0,0,-1)$ & 125, 36, 47\\ \hline
68& $ (-1,-1,1,1), (0,0,-1,0), (0,0,-1,-1)$ &125, 36, 47 \\ \hline
97& $(-1,0,0,1), (0,-1,0,1), (0,0,-1,-1)$ &347, 15, 26 \\ \hline
109& $(-1,0,0,1), (0,-1,1,0), (0,0,-1,-1)$ &347, 15, 26\\ \hline
117& $(-1,0,0,1), (0,0,1,-1), (0,-1,-1,0)$ &237, 15, 46\\ \hline
129& $(-1,0,0,1), (0,-1,-1,1), (0,0,0,-1)$ &236, 15, 47\\ \hline
132& $ (-1,0,0,1), (0,-1,0,0), (0,1,-1,-1)$ &347, 15, 26\\ \hline
133& $(-1,0,0,1), (0,-1,0,0), (0,0,-1,-1)$ &347, 15, 26\\ \hline
134& $(-1,0,0,1), (0,0,0,-1), (1,-1,-1,0)$ &237, 15, 46 \\ \hline
135& $(-1,0,0,1), (0,0,0,-1), (0,-1,-1,0)$ &237, 15, 46\\ \hline
136& $(-1,0,0,1), (0,0,0,-1), (0,-1,-1,-1)$ &237, 15, 46\\ \hline
140& $(-1,-1,0,1), (0,0,-1,0), (0,0,0,-1)$ &125, 36, 47\\ \hline
143& $(-1,0,0,0), (0,-1,0,0), (0,0,-1,-1)$ &347, 15, 26\\ \hline
\end{tabular}
\medskip
\caption{Vertices and minimal nonfaces of $P$ with $(V(P),F(P))=(7,12)$} \label{tab:VMNF712}
\end{table}
}

We see that 
\begin{equation} \label{eq:diff712}
X_{30}\cong X_{43}, \qquad X_{68}\cong X_{134},\qquad X_{129}\cong X_{136},
\end{equation} 
see Subsection~\ref{subsec:diffeo}. Using the data in Table~\ref{tab:VMNF712}, we obtain the following table. 

{\scriptsize
\begin{table}[H]
\centering
\begin{tabular}{|c|c|c|c|c|c|c|c|c|} \hline
ID & $\mathcal{I}$ & s.v.e. & s.v.e. $\ZZ/2$ & $4$-v.e. $\ZZ/2$ & c.v.e. $\ZZ/3$ & c.v.e. $\ZZ/2$ \\ \hline
\begin{tabular}{c}30\\
(43)\end{tabular} & $x^3,y^2,z(2x+y-z)$ & $y$ & $y$ & all & $(x,y)$ & \begin{tabular}{c}$x,y,z,$\\
$y+z$\end{tabular} \\ \hline
31 & $x^3,y(2x-y),z(y-z)$ & $\emptyset$ & $y$ & & & \\ \hline
35 & $x^2(x-y),y^2,z(2x-z)$& $y$ & $(y,z)$ & all & $y$ & \\ \hline
42 & $x^3,y^2,z(2x-z)$ & $y$ & $(y,z)$ & all & $(x,y)$ & \\ \hline
49 & $x^3,y(x-y),z(x+y-z)$ & $\emptyset$ & $\emptyset$ & $(x,y)$ & $x$ & $x$ \\ \hline
66 & $x^3,y(x-y),z(x-z)$ & $\emptyset$ & $\emptyset$ & all & & \\ \hline
\begin{tabular}{c}68\\
(134)\end{tabular} & $x^3,y(x-y-z),z(x-z)$& $\emptyset$ & $\emptyset$ & $(x,z)$ & $x$ & $x$ \\ \hline
97 & $x^2,y^2,z^2(x+y-z)$ & $x,y$ & $(x,y)$ & all & $(x,y)$ &\begin{tabular}{c} $x,y,x+y,$\\
$x+y+z$ \end{tabular} \\ \hline
109 & $x^2,y^2,z(x-z)(y-z)$ & $x,y$ & $(x,y)$ & $(x,y)$ & &\\ \hline
117 & $x^2, y(x-y),z^2(y-z)$ & $x,x-2y$ & $x$ & $(x,y)$ & & \\ \hline
\begin{tabular}{c}129\\
(136)\end{tabular} & $x^2,y^3,z(x+y-z)$ & $x$ & $x$ & $(x,y)$ & $(x,y)$ & \\ \hline
132 & $x^2,y(y-z),z^2(x-z)$ & $x$ & $x$ & $(x,z)$ & $x$ & \\ \hline
133 & $x^2,y^2,z^2(x-z)$ & $x,y$ & $(x,y)$ & all & $(x,y)$ & \begin{tabular}{c}$x,y,x+y$\\
$x+z$\end{tabular}\\ \hline
135 & $x^2,y(x-y),z^3$ & $x,x-2y$ & $x$ & all & & \\ \hline
140 & $x^3,y^2,z(x-z)$ & $y$ & $y$ & all & $(x,y)$ & $x,y$ \\ \hline
143 & $x^2,y^2,z^3$ & $x,y$ & $(x,y)$ & all & all & \\ \hline
\end{tabular}
\medskip
\caption{Ideals and invariants when $(V(P), F(P))=(7,12)$}
\label{d=4v=7f=12}
\end{table}
}

Table~\ref{d=4v=7f=12} shows that the cohomology rings in the table are distinguished by the invariants in the table except for two pairs: $49$ and $68(134)$, $97$ and $133$. We shall prove that their cohomology rings are not isomorphic to each other. 

The c.v.e. of $H^*(X_{49})$ and $H^*(X_{68})$ are both $x$ up to sign. Therefore, if there is an isomorphism $H^*(X_{49}) \rightarrow H^*(X_{68})$, then it induces an isomorphism $H^*(X_{49})/(x^2) \rightarrow H^*(X_{68})/(x^2)$. However, Table~\ref{4968} shows that this does not occur, so $H^*(X_{49})\not\cong H^*(X_{68})$. 

{\scriptsize
\begin{table}[H]
\centering
\begin{tabular}{|c|c|c|c|} \hline
ID & $\mathcal{I}+(x^2)$ & c.v.e. $\ZZ/3$ \\ \hline
49 & $x^2,y(x-y),z(x+y-z)$ & all \\ \hline
68(134) & $x^2,y(x-y-z), z(x-z)$ & $(x,z)$ \\ \hline
\end{tabular}
\medskip
\caption{Distinguishment between $H^*(X_{49})$ and $H^*(X_{68})$}
\label{4968}
\end{table}
}

The s.v.e. of $H^*(X_{97})$ and $H^*(X_{133})$ are both $x,y$ up to sign and the transposition of $x$ and $y$ induces an automorphism of $H^*(X_{97})$ since the ideal $\mathcal{I}$ of $H^*(X_{97})$ is invariant under the transposition. Therefore, if there is an isomorphism $F\colon H^*(X_{97})\to H^*(X_{133})$, then we may assume that $F(x)=\pm x$, so that $F$ induces an isomorphism $: H^*(X_{97})/(x) \rightarrow H^*(X_{133})/(x)$. However, Table~\ref{97133} shows that this does not occur, so $H^*(X_{97})\not\cong H^*(X_{133})$. 

{\scriptsize
\begin{table}[H]
\centering
\begin{tabular}{|c|c|c|c|} \hline
ID & $\mathcal{I}+(x)$ & c.v.e. $\ZZ/3$ \\ \hline
97& $y^2, z^2(y-z)$ & $y$ \\ \hline
133 & $y^2,z^3$ & all \\ \hline
\end{tabular}
\medskip
\caption{Distinguishment between $H^*(X_{49})$ and $H^*(X_{68})$}
\label{97133}
\end{table}
}


\subsubsection{$(V(P),F(P))=(7,13)$}

In this case, there are six smooth Fano $4$-polytopes. 
We shall prove that these six cohomology rings are not isomorphic to each other. 

{\scriptsize
\begin{table}[H]
\centering
\begin{tabular}{|c|c|c|} \hline
ID & vertices of $P$ from $5$th & minimal nonfaces \\ \hline
40 & $ (-1,-1,0,2), (0,1,0,-1), (0,0,-1,-1) $ & 125, 156, 237, 347, 46\\ \hline
41 & $(-1,-1,0,2), (1,1,-1,-1), (0,0,0,-1) $ & 125, 127, 346, 356, 47\\ \hline
60 & $(-1,-1,1,1), (0,1,-1,0), (0,0,-1,-1) $ & 125, 156, 247, 347, 36\\ \hline
64 & $(-1,-1,1,1), (1,1,-1,-1), (-1,-1,0,0) $ & 125, 127, 346, 347, 56\\ \hline
69 & $(-1,-1,1,1), (0,1,-1,-1), (0,0,-1,-1) $ & 125, 156, 346, 347, 27\\ \hline
137& $(-1,0,0,1), (1,-1,0,-1), (-1,0,-1,0) $ & 137, 246, 256, 347, 15\\ \hline
\end{tabular}
\medskip
\caption{Vertices and minimal nonfaces of $P$ with $(V(P),F(P))=(7,13)$} \label{tab:VMNF713}
\end{table}
}

Using the data in Table~\ref{tab:VMNF713}, we obtain the following table. 

{\scriptsize
\begin{table}[H]
\centering
\begin{tabular}{|c|c|c|c|} \hline
ID & $\mathcal{I}$ & c.v.e. & c.v.e. $\ZZ/2$ \\ \hline
40 & $x^3,x^2y,z^2(x-y),z^2(x-z), y(2x-y-z)$ & $x$ & $x,y+z$ \\ \hline
41 & $x(x-y)^2,y(x-y)^2,y^2(y+z),xy^2, z(2x-y-z)$ & $\emptyset$ & \\ \hline
60 & $x^3,x^2y,z(x-y)(x-z),z^2(x-z), y(x-y-z)$ & $x$ & $x$ \\ \hline
64 & $x(x+z)^2,z(x-y+z)^2,y^3,z(x-y)^2, xy$ & $y$ & $y,z$ \\ \hline
69 & $x^3,x^2y,y(x-y-z)^2,z^3, z(x-y)$ & $x,z$ & $x,z$ \\ \hline
137 & $z^3,y^3,xy^2,z^2(x-y), x(x-y+z)$ & $y,z$ & $y,z$\\ \hline
\end{tabular}
\medskip
\caption{Ideals and invariants when $(V(P),F(P))=(7,13)$}
\label{d=4v=7f=13}
\end{table}
}

Table~\ref{d=4v=7f=13} shows that the six cohomology rings can be distinguished by the invariants in the table except for two pairs: $40$ and $64$, $69$ and $137$. We shall prove that their cohomology rings are not isomorphic to each other. 

The c.v.e. of $H^*(X_{40})$ and $H^*(X_{64})$ are respectively $x$ aand $y$ up to sign. Therefore, if there is an isomorphism $H^*(X_{40})\to H^*(X_{64})$, then it induces an isomorphism $H^*(X_{40})/(x)\to H^*(X_{64})/(y)$. However, this does not occur because 
\[
\begin{split}
H^*(X_{40})/(x)&=\ZZ[y,z]/(yz^2, z^3, y(y+z)),\\
H^*(X_{64})/(y)&=\ZZ[x,z]/(x(x+z)^2,z(x+z)^2, x^2z),
 \end{split}
\]
and the degree sequences of these ideals are different. Therefore $H^*(X_{40})\not\cong H^*(X_{64})$. 

Similarly, the c.v.e. of $H^*(X_{69})$ and $H^*(X_{137})$ are respectively $x,z$ and $y,z$ up to sign.  Therefore, if there is an isomorphism $H^*(X_{69})\to H^*(X_{137})$, then it induces an isomorphism $H^*(X_{69})/(x,z)\to H^*(X_{137})/(y,z)$. However, this does not occur because  
\[
H^*(X_{69})/(x,z)=\ZZ[y]/(y^3),\qquad H^*(X_{137})/(y,z)=\ZZ[x]/(x^2),
\]
and these quotient rings are not isomorphic. Therefore $H^*(X_{69})\not\cong H^*(X_{137})$.


\subsection{The case where $V(P)=8$} 

\subsubsection{$(V(P),F(P))=(8,15)$}

In this case, there are ten smooth Fano $4$-polytopes and they are all combinatorially equivalent to a direct sum of a $5$-gon and a $2$-simplex.

{\scriptsize
\begin{table}[H]
\centering
\begin{tabular}{|c|c|c|} \hline
ID & vertices of $P$ from $5$th & minimal nonfaces \\ \hline
26 & $(-1,-1,0,2), (0,0,-1,1), (0,0,1,-1), (0,0,-1,0)$ & 125, 36, 38, 47, 48, 67\\ \hline
28 & $(-1,-1,0,2), (0,0,-1,1), (0,0,1,-1), (0,0,0,-1)$ &125, 36, 38, 47, 48, 67 \\ \hline
32 & $(-1,-1,0,2), (0,0,1,-1), (0,0,-1,0), (0,0,0,-1)$ & 125, 37, 38, 46, 48, 67\\ \hline
45 & $ (-1,-1,1,1), (0,0,-1,1), (0,0,1,-1), (0,0,-1,0)$ & 125, 36, 38, 47, 48, 67\\ \hline
48 & $(-1,-1,1,1), (0,0,-1,1), (0,0,-1,0), (0,0,0,-1)$ & 125, 36, 37, 47, 48, 68\\ \hline
67 & $(-1,-1,1,1), (0,0,-1,0), (0,0,0,-1), (0,0,-1,-1)$ & 125, 36, 38, 47, 48, 67\\ \hline
118 & $(-1,0,0,1), (1,0,0,-1), (-1,0,0,0), (0,-1,-1,1)$ & 238, 15, 17, 46, 47, 56\\ \hline
123 & $ (-1,0,0,1), (1,0,0,-1), (-1,0,0,0), (1,-1,-1,0)$ & 238, 15, 17, 46, 47, 56\\ \hline
124 & $(-1,0,0,1), (1,0,0,-1), (-1,0,0,0), (0,-1,-1,0)$ & 238, 15, 17, 46, 47, 56\\ \hline
125 & $(-1,0,0,1), (1,0,0,-1), (-1,0,0,0), (1,-1,-1,-1)$ & 238, 15, 17, 46, 47, 56\\ \hline
\end{tabular}
\medskip
\caption{Vertices and minimal nonfaces of $P$ with $(V(P),F(P))=(8,15)$} \label{tab:VMNF815}
\end{table}
}

We see that 
\begin{equation} \label{eq:diff815}
X_{28}\cong X_{32},\qquad X_{67}\cong X_{118}, \qquad X_{123}\cong X_{125},
\end{equation}
see Subsection~\ref{subsec:diffeo}. We obtain the following table from Table~\ref{tab:VMNF815}. 

{\scriptsize
\begin{table}[H]
\centering
\begin{tabular}{|c|c|c|c|c|c|} \hline
ID & $\mathcal{I}$ & s.v.e. & s.v.e. $\ZZ/2$ & c.v.e. \\ \hline
26 & \begin{tabular}{c}$x^3, y(y+u),u(y-z+u),$\\
$z(2x-z),u(2x-u),yz$\end{tabular} & $\emptyset$ & $(z,u)$ &\\ \hline
\begin{tabular}{c} 28\\
(32)\end{tabular} &\begin{tabular}{c} $x^3, y^2,u(y-z),$\\
$z(2x-z-u),u(2x-u),yz$\end{tabular} & $y$ & $(y,u)$ & \\ \hline
45 & \begin{tabular}{c}$x^3, y(x-y-u),u(2x-u),$\\
$x(x-z), u(x+y-z),yz$\end{tabular} & $\emptyset$ & $u$ & \\ \hline
48 & \begin{tabular}{c}$x^3, y(x-y-z),z(x-y-z),$\\
$z(x+y-u),u(x-u),yu$\end{tabular} & $\emptyset$ & $\emptyset$ & $x$ \\ \hline
\begin{tabular}{c}67\\
(118)\end{tabular} & \begin{tabular}{c}$x^3, y(x-y-u),u(x-y-u),$\\
$z(x-z-u),u(y-z),yz$\end{tabular} & $\emptyset$ & $\emptyset$ & $x$ \\ \hline
\begin{tabular}{c}123\\
(125)\end{tabular} & \begin{tabular}{c}$u^3, x(x+z-u),z(z-u),$\\
$y^2,z(x-y),xy$\end{tabular} & $y$ & $y$ & \\ \hline
124 & $u^3, x(x+z),z^2,y^2,z(x-y),xy$ & $\infty$ & &\\ \hline
\end{tabular}
\medskip
\caption{Ideals and invariants when $(V(P),F(P))=(8,15)$}\label{d=4v=8f=15}
\end{table}
}

Tables~\ref{d=4v=8f=15} shows that the seven cohomology rings in the table can be distinguished by the invariants in the table except for $48$ and $67(118)$. The c.v.e. of $H^*(X_{48})$ and $H^*(X_{67})$ are both $x$ up to sign. Therefore, if there is an isomorphism $H^*(X_{48})\to H^*(X_{67})$, then it induces an isomorphism $H^*(X_{48})/(x^2)\to H^*(X_{67})/(x^2)$. However, Table~\ref{4867} shows that this does not occur, so $H^*(X_{48})\not\cong H^*(X_{67})$. 

{\scriptsize
\begin{table}[H]
\begin{tabular}{|c|c|c|} \hline
ID & $\mathcal{I}+(x^2)$ & c.v.e. $\ZZ/3$ \\ \hline
48 & $x^2, y(x-y-z),z(x-y-z),z(x+y-u),u(x-u),yu$ & $(x,y,z,u)$ \\ \hline
67(118) & $x^2, (y(x-y-u),u(x-y-u),z(x-z-u),u(x-z-u),yz$ & $(x,y+u,z+u)$ \\ \hline
\end{tabular}
\medskip
\caption{Distinguishment between $H^*(X_{48})$ and $H^*(X_{67})$}\label{4867}
\end{table}
}

\subsubsection{$(V(P),F(P))=(8,16)$}

In this case, there are 28 smooth Fano $4$-polytopes and there are two combinatorial types among them. Indeed, 13 polytopes among them have the same combinatorial type as a cross-polytope as shown in Table~\ref{tab:VMNF816_1}. The corresponding toric Fano $4$-folds are Bott manifolds. 
The cohomology rings associated to these two different combinatorial types are not isomorphic to each other because the degree sequences of their ideals are different, see Tables~\ref{d=4v=8f=16Bott} and~\ref{d=4v=8f=16}.  More generally, it is known that a toric manifold which has the same cohomology ring as a Bott manifold is indeed a Bott manifold (\cite{MP08}).  

{\scriptsize
\begin{table}[H]
\begin{tabular}{|c|c|c|} \hline
ID & vertices of $P$ from $5$th & minimal nonfaces \\ \hline 
74 & $(-1,0,0,1), (0,-1,0,1), (0,0,-1,1), (0,0,0,-1)$ &15, 26, 37, 48 \\ \hline
75& $(-1,0,0,1), (0,-1,0,1), (0,0,1,-1), (0,0,-1,0)$ & 15, 26, 38, 47\\ \hline
83 & $(-1,0,0,1), (0,-1,0,1), (0,1,-1,0), (0,0,0,-1)$ &15, 26, 37, 48 \\ \hline
95 & $(-1,0,0,1), (0,-1,0,1), (0,0,-1,0), (0,0,0,-1)$ & 15, 26, 37, 48\\ \hline
96 & $(-1,0,0,1), (0,-1,0,1), (0,0,0,-1), (0,0,-1,-1)$ & 15, 26, 38, 47\\ \hline
105 & $(-1,0,0,1), (0,-1,1,0), (0,1,0,-1), (0,0,-1,0)$ & 15, 26, 38, 47\\ \hline
106 & $(-1,0,0,1), (0,-1,1,0), (0,0,-1,0), (0,0,0,-1)$ & 15, 26, 37, 48\\ \hline
108 & $ (-1,0,0,1), (0,-1,1,0), (0,0,-1,0), (0,0,-1,-1)$ & 15, 26, 37, 48\\ \hline
112 & $(-1,0,0,1), (0,0,1,-1), (0,-1,0,0), (0,0,-1,0)$ & 15, 27, 38, 46\\ \hline
114 & $(-1,0,0,1), (0,0,1,-1), (0,-1,1,-1), (0,0,-1,0)$ & 15, 27, 38, 46\\ \hline
130 & $(-1,0,0,1), (0,-1,0,0), (0,0,-1,0), (0,0,0,-1)$ & 15, 26, 37, 48\\ \hline
131 & $(-1,0,0,1), (0,-1,0,0), (0,0,0,-1), (0,0,-1,-1)$ & 15, 26, 38, 47\\ \hline
142 & $(-1,0,0,0), (0,-1,0,0), (0,0,-1,0), (0,0,0,-1)$ & 15, 26, 37, 48\\ \hline
\end{tabular}
\medskip
\caption{Vertices and minimal nonfaces of $P$ with $(V(P),F(P))=(8,16)$, cross-polytope} \label{tab:VMNF816_1}
\end{table}
}

One can see that 
\begin{equation} \label{eq:diff816-1}
X_{74}\cong X_{96}, \qquad X_{83}\cong X_{108}, \qquad X_{95}\cong X_{131},
\end{equation}
see Subsection~\ref{subsec:diffeo}. Using the data in Table~\ref{tab:VMNF816_1}, we obtain the following table. 

{\scriptsize
\begin{table}[H]
\centering
\begin{tabular}{|c|c|c|c|c|} \hline
ID & $\mathcal{I}$ & s.v.e. & $4$-v.e. $\ZZ/2$ \\ \hline
74(96) & $x^2,y^2,z^2,u(x+y+z-u)$ & $x,y,z$ & all \\ \hline
75 & $x^2,y^2,u(z-u),z(x+y-z)$ & $x,y$ & all \\ \hline
83(108)& $x^2,y(y-z),z^2,u(x+y-u)$ & $x,z,z-2y$ & $(x,y,z)$ \\ \hline
95(131)& $x^2,y^2,z^2,u(x+y-u)$ & $x,y,z$ & all \\ \hline
105& $x^2,y(y-z),u(y-u),z(x-z)$ & $x,x-2z$ & $(x,y,z)$ \\ \hline
106& $x^2,y^2,z(y-z),u(x-u)$ & $x,y,x-2u,y-2z$ &  \\ \hline
112& $x^2,z^2,u(y-u),y(x-y)$ & $x,z,x-2y$ & all \\ \hline
114& $x^2,z^2,u(y+z-u),y(x-y-z)$ & $x,z$ & $(x,y,z)$ \\ \hline
130& $x^2,y^2,z^2,u(x-u)$ & $x,y,z,x-2u$ &  \\ \hline
142& $x^2,y^2,z^2,u^2$ & $x,y,z,u$ &  \\ \hline
\end{tabular}
\medskip
\caption{Ideals and invariants when $(V(P),F(P))=(8,16)$, cross-polytope}\label{d=4v=8f=16Bott}
\end{table}
}

Table~\ref{d=4v=8f=16Bott} shows that the cohomology rings in the table can be distinguished by the invariants in the table except for $74(96)$ and $95$.

Suppose that there is an isomorphism $F : H^*(X_{74}) \rightarrow H^*(X_{95})$. Since the s.v.e. of $H^*(X_{74})$ and $H^*(X_{95})$ are $x,y,z$ and any permutation of $x,y,z$ is an automorphism of $H^*(X_{74})$, we may assume that $F(x) = \pm x$, $F(y) = \pm y$, $F(z) = \pm z$. Therefore, $F$ induces an isomorphism $H^*(X_{74})/(x,y)\cong H^*(X_{95})/(x,y)$. However, this does not occur because 
\[
H^*(X_{74})/(x,y)=\ZZ[z,u]/(z^2,u(u-z)),\qquad H^*(X_{95})=\ZZ[z,u]/(z^2,u^2)
\]
and these two rings are not isomorphic (e.g. their s.v.e. are different). Therefore, $H^*(X_{74})\not\cong H^*(X_{95})$.


Next, we shall treat the other case where $P$ is not combinatorially equivalent to a cross-polytope. There are 15 smooth Fano $4$-polytopes in this case as shown in Table~\ref{tab:VMNF816_2}. 


{\scriptsize
\begin{table}[H]
\centering
\begin{tabular}{|c|c|c|} \hline
ID & vertices of $P$ from $5$th & minimal nonfaces \\ \hline
29 & $(-1,-1,0,2), (0,0,-1,1), (0,1,0,-1), (0,0,0,-1)$ & 125, 157, 28, 36, 47, 48\\ \hline
33 & $(-1,-1,0,2), (0,1,-1,0), (0,1,0,-1), (0,0,0,-1)$ & 125, 157, 28, 36, 47, 48\\ \hline
34 & $(-1,-1,0,2), (0,1,-1,0), (1,0,0,-1), (0,0,0,-1)$ & 125, 257, 18, 36, 47, 48 \\ \hline
37 & $(-1,-1,0,2), (0,1,0,-1), (0,0,-1,0), (0,0,0,-1)$ & 125, 156, 28, 37, 46, 48\\ \hline
38 & $(-1,-1,0,2), (0,1,0,-1), (0,0,0,-1), (0,1,-1,-1)$ & 125, 156, 27, 38, 46, 47\\ \hline
39 & $(-1,-1,0,2), (0,1,0,-1), (0,0,0,-1), (0,0,-1,-1)$ & 125, 156, 27, 38, 46, 47\\ \hline
47 & $(-1,-1,1,1), (0,0,-1,1), (0,1,0,-1), (0,0,0,-1)$ & 125, 157, 28, 36, 47, 48\\ \hline
59 & $ (-1,-1,1,1), (0,1,-1,0), (0,0,-1,0), (0,0,0,-1)$ & 125, 156, 27, 36, 37, 48\\ \hline 
93 & $(-1,0,0,1), (0,-1,0,1), (0,1,0,-1), (0,-1,-1,0)$ & 238, 348, 15, 26, 47, 67\\ \hline
94 & $(-1,0,0,1), (0,-1,0,1), (0,1,0,-1), (0,0,-1,-1)$ & 238, 348, 15, 26, 47, 67\\ \hline
104 & $(-1,0,0,1), (0,-1,1,0), (0,1,-1,0), (0,-1,0,-1)$ & 248, 348, 15, 26, 37, 67\\ \hline
111 & $(-1,0,0,1), (0,0,1,-1), (0,-1,-1,1), (0,0,0,-1)$ & 237, 267, 15, 38, 46, 48\\ \hline
115 & $(-1,0,0,1), (0,0,1,-1), (0,0,0,-1), (1,-1,-1,0)$ & 238, 268, 15, 37, 46, 47\\ \hline
116 & $(-1,0,0,1), (0,0,1,-1), (0,0,0,-1), (0,-1,-1,0)$ & 238, 268, 15, 37, 46, 47\\ \hline
126 & $(-1,0,0,1), (1,0,0,-1), (0,-1,0,0), (-1,0,-1,0)$ & 138, 348, 15, 27, 46, 56\\ \hline 
\end{tabular}
\medskip
\caption{Vertices and minimal nonfaces of $P$ with $(V(P),F(P))=(8,16)$, non-cross-polytope } \label{tab:VMNF816_2}
\end{table}
}

We see that 
\begin{equation} \label{eq:diff816-2}
X_{29}\cong X_{39},\qquad X_{111}\cong X_{116},
\end{equation}
see Subsection~\ref{subsec:diffeo}. Using the data in Table~\ref{tab:VMNF816_2}, we obtain the following table. 

{\scriptsize
\begin{table}[H]
\centering
\begin{tabular}{|c|c|c|c|c|c|} \hline
ID & $\mathcal{I}$ & s.v.e. &s.v.e.$\ZZ/2$& $4$-v.e.$\ZZ/2$ & c.v.e. $\ZZ/3$ \\ \hline
\begin{tabular}{c}29\\
(39)\end{tabular} & \begin{tabular}{c}
$x^3,x^2z, u(x-z),y^2$,\\
$z(2x+y-z-u),u(x+y-u)$ 
\end{tabular} & $y$ &$y$ & $(x,y,z+u)$ & $(x,y)$ \\ \hline
33 & \begin{tabular}{c}
$x^2(x-y),x^2z,u(x-y-z),y^2,$\\
$z(2x-z-u),u(2x-z-u)$
\end{tabular}
& $y$ & $(y,z+u)$ & all & $y$ \\ \hline
34 & \begin{tabular}{c}
$x^2(x-y),xz(x-y), u(x-z),$\\
$y^2,z(2x-z-u),u(x-u)$ 
\end{tabular}
& $y$ & $(y,z+u)$ & $(x,y,z+u)$ & \\ \hline
37 & \begin{tabular}{c}
$x^3,x^2y, u(x-y),z^2,$\\
$y(2x-y-u),u(x-u)$ 
\end{tabular}
& $z$ & $(y+u,z)$ & all & $(x,y,z)$\\ \hline
38 & \begin{tabular}{c}
$x^2(x-u),x^2y, z(x-y-u),u^2,$\\
$y(2x-y-z-u),z(x-z)$ 
\end{tabular}& $u$ & $u$ & $(x,y+z,u)$ & $u$\\ \hline
47 & \begin{tabular}{c} $x^3,x^2z, u(x-z),y(x-y),$\\
$z(x+y-z-u),u(y-u)$ \end{tabular}
& $\emptyset$& $\emptyset$ & $(x,y,z)$ & $x$\\ \hline
59 & \begin{tabular}{c}
$x^3,x^2y, z(x-y),y(x-y-z),$\\
$z^2,u(x-u)$ 
\end{tabular} & $z$& $z$ & all & \\ \hline
93 & \begin{tabular}{c} $u^2(y-z+u),u^2(x-u),$\\
$ x^2,y(y+u),z(x-z),yz$ 
\end{tabular} & $x,x-2z$ & $x$ & $(x,z,u)$ & $(x,z)$\\ \hline
94 & \begin{tabular}{c}$u^2(y-z),u^2(x-u), x^2,$\\
$y^2,z(x-z-u),yz$
\end{tabular} & $x,y$& $(x,y)$ & all & $(x,y)$\\ \hline
104& \begin{tabular}{c} $u^2(x-u),u(y-z)(x-u), $\\
$x^2,y(y+u),z^2,yz$
\end{tabular} & $x,z$& $(x,z)$ & $(x,z,u)$ & \\ \hline
\begin{tabular}{c}111\\
(116)\end{tabular}& \begin{tabular}{c}$z^3,yz^2, x^2,u(y-z),$\\
$y(x-y+z-u),u(x-u)$ 
\end{tabular} & $x, x-2u$ & $x$ & $(x,z,u)$ & $(x,z,u)$\\ \hline
115& \begin{tabular}{c}$u^3,yu^2, x(x-u),z(y-u),$\\
$y(x-y-z),z(x-y-z)$
\end{tabular} & $\emptyset$& $\emptyset$ & $(x,y,u)$ & $u$\\ \hline
126& \begin{tabular}{c}$u^3,u^2(x-y), x(x+u),$\\
$z^2,y^2,xy$
\end{tabular} & $y,z$& $(y,z)$ & all & $(y,z,u)$\\ \hline
\end{tabular}
\medskip
\caption{Ideals and invariants when $(V(P), F(P))=(8,16)$, non-cross-polytope}\label{d=4v=8f=16}
\end{table}
}

Table~\ref{d=4v=8f=16} shows that the cohomology rings in the table are distinguished by the invariants in the table except for $47$ and $115$. The c.v.e. of $H^*(X_{47})$ and $H^*(X_{115})$ are respectively $x$ and $u$ up to sign. Therefore, if there is an isomorphism $H^*(X_{47})\to H^*(X_{115})$, then it induces an isomorphism $H^*(X_{47})/(x^2)\to H^*(X_{115})/(u^2)$. However, Table~\ref{47115} shows that this does not occur, so $H^*(X_{47}) \not\cong H^*(X_{115})$. 

{\scriptsize
\begin{table}[H]
\begin{tabular}{|c|c|c|c|} \hline
ID & $\mathcal{I}+((\text{c.v.e.})^2)$ & c.v.e. $\ZZ/3$ \\ \hline
47 & $x^2, u(x-z),y(x-y),z(x+y-z-u),u(y-u)$ & $(x,y,z+u)$\\ \hline
115 & $u^2, x(x-u),z(y-u),y(x-y-z),z(x-y-z)$ & $(x,u)$\\ \hline
\end{tabular}
\medskip
\caption{Distinguishment between $H^*(X_{47})$ and $H^*(X_{115})$}\label{47115}
\end{table}
}

\subsubsection{$(V(P),F(P))=(8,17)$}

There are seven smooth Fano $4$-polytopes in this case and there are two combinatorial types among them: one is ID numbers $36, 65$ and the other one is ID numbers $50, 57, 58, 61, 110$, see Table~\ref{tab:VMNF817}. One can distinguish the cohomology rings of the former class between those of the latter class by the degree sequences of the ideals.  Therefore, Table~\ref{d=4v=8f=17} shows that the seven cohomology rings can be distinguished by the degree sequences of the ideals and c.v.e. over $\ZZ/2$ except for 50 and 57.  

The fans of $X_{50}$ and $X_{57}$ do not satisfy the condition in Lemma~\ref{diffeolemma} up to unimodular equivalence, so they are not \emph{weakly equivariantly} diffeomorphic with respect to the restricted $(S^1)^4$-actions. However, their cohomology rings are isomorphic to each other.  Indeed, the map 
\begin{equation} \label{eq:iso5057}
(x,y,z,u)\to (-x+2u,\ -y+u,\ u,\ -z)
\end{equation}
gives an isomorphism $H^*(X_{50})\to H^*(X_{57})$.  It does not preserve their total Chern classes (even their first Chern classes) but it does preserve their total Pontryagin classes (in fact, the first and second Pontryagin classes because of dimensional reason) which  are given by 
{\small
\[
\begin{split}
p(X_{50})&=(1+x^2)^2(1+(x-y-z+u)^2)(1+(x-y)^2)(1+(x-z)^2)(1+y^2)(1+z^2)(1+u^2),\\ 
p(X_{57})&=(1+x^2)^2(1+(x-y+u)^2)(1+(x-y)^2)(1+(x-z)^2)(1+y^2)(1+z^2)(1+u^2),
\end{split}
\]
}
see Remark~\ref{rem:pont}.

{\scriptsize
\begin{table}[H]
\centering
\begin{tabular}{|c|c|c|} \hline
ID & vertices of $P$ from $5$th & minimal nonfaces \\ \hline
36 & $(-1,-1,0,2), (0,1,0,-1),(0,-1,-1,1), (0,0,0,-1)$ & 28, 46, 48, 125, 156, 158, 237, 347, 367\\ \hline
50 & $(-1,-1,1,1), (0,1,-1,0),(0,1,0,-1), (0,-1,0,0)$ & 28, 36, 47, 125, 156, 157, 348 \\ \hline
57 & $(-1,-1,1,1), (0,1,-1,0),(0,-1,0,0), (0,0,0,-1)$ & 27, 36, 48, 125, 156, 158, 347\\ \hline
58 & $(-1,-1,1,1), (0,1,-1,0),(0,-1,0,0), (0,1,-1,-1)$ & 27, 36, 48, 125, 156, 158, 347 \\ \hline
61 & $(-1,-1,1,1), (1,1,-1,-1),(-1,0,0,0), (0,-1,0,0)$ & 17, 28, 56, 125, 346, 347, 348 \\ \hline
65 & $(-1,-1,1,1), (1,1,-1,-1),(-1,-1,0,0), (0,0,-1,-1)$ & 56, 58, 67, 125, 127, 128, 346, 347, 348\\ \hline
110& $(-1,0,0,1), (0,0,1,-1),(1,0,-1,0), (-1,-1,0,0)$ & 15, 37, 46, 128, 238, 248, 567 \\ \hline
\end{tabular}
\medskip
\caption{$(V(P),F(P))=(8,17)$} \label{tab:VMNF817} 
\end{table}
\begin{table}[H]
\centering
\begin{tabular}{|c|c|c|} \hline
ID & $\mathcal{I}$ & c.v.e. $\ZZ/2$ \\ \hline
36 & $(x-y+z)u,(2x-y+z-u)y,(x-u)u,x^2(x+z)$, & $x+z,(y+z,u)$ \\
& $x^2y,x^2u,(x+z)z^2,z^2(x-u),z^2y$ & \\ \hline
65 & $xy,xu,yz,(x+z)^2x,(x+z)^2z$, & $z+u,y+u,x+z$ \\
& $(y-z)^2u,(x-y-u)^2y,(x-u)^2z,(y+u)^2u$ & \\ \hline\hline
50 & $(x-y-z+u)u,(x-y)y,(x-z)z$, & $(y,z),x+u$ \\
& $x^2(x+u),x^2y,x^2z,(x-y)(x-z)u$ & \\ \hline
57 & $(x-y+z)z,(x-y)y,(x-u)u$, & $(y,u),x+z$ \\
& $x^2(x+z),x^2y,x^2u,(x-y)(x-u)z$ & \\ \hline
58 & $z(x-y+z-u),y(x-y-u),u(x-u)$, & $z,u,y+u$ \\
& $x^2(x+z),x^2y,x^2u,z(x-y)(x-u)$ & \\ \hline
61 & $z(x-y+z),u(x-y+u),xy$, & $(y,z,u)$ \\
& $x(x+z)(x+u),y^3,z^3,u^3$ & \\ \hline
110 & $(x-z+u)x,(y-z)z,(x-y)y$, & $z,u$ \\
& $u^3,u^2(y-z),u^2(x-y),xyz$ & \\ \hline
\end{tabular}
\medskip
\caption{Ideals and invariants when $(V(P), F(P))=(8,17)$}
\label{d=4v=8f=17}
\end{table}
}


\subsubsection{$(V(P),F(P))=(8,18)$}

In this case, there are two smooth Fano polytopes as shown in Table~\ref{tab:VMNF818}. 
Their cohomology rings can be distinguished by c.v.e. over $\ZZ/3$ as shown in the following table. 

{\scriptsize
\begin{table}[H]
\centering
\begin{tabular}{|c|c|c|} \hline
ID & vertices of $P$ from $5$th & minimal nonfaces \\ \hline
53 & $(-1,-1,1,1), (0,1,-1,0), (1,0,0,-1), (-1,-1,0,0)$ & 36, 47, 125, 128, 138, 156, 248, 257, 348, 567\\  \hline
55 & $(-1,-1,1,1), (0,1,-1,0), (1,0,0,-1), (0,0,-1,-1)$ & 36, 47, 125, 128, 138, 156, 248, 257, 348, 567\\\hline
\end{tabular}
\medskip
\caption{$(V(P),F(P))=(8,18)$} \label{tab:VMNF818} 
\end{table}
\begin{table}[H]
\centering
\begin{tabular}{|c|c|c|} \hline
ID & $\mathcal{I}$ & c.v.e. $\ZZ/3$ \\ \hline
53 & $(x-y)y,(x-z)z,(x+u)^2x,u^3,(x-y)u^2$, & $u$ \\ 
& $(x+u)xy,(x-z)u^2,(x+u)xz,(x-y)(x-z)u,xyz$ & \\ \hline
55 & $(x-y-u)y,(x-z-u)z,x^3,(x-y)(x-z)u,u^3$, & $(x,u)$ \\
& $x^2y,u^2(x-z),x^2z,u^2(x-y),xyz$ & \\ \hline
\end{tabular}
\medskip
\caption{Ideals and invariants when $(V(P),F(P))=(8,18)$}
\label{d=4v=8f=18}
\end{table}
}


\subsection{The case where $V(P)=9$}

\subsubsection{$(V(P),F(P))=(9,18)$}

In this case, there are four smooth Fano $4$-polytopes as shown in Table~\ref{tab:VMNF918} and they are all combinatorially equivalent to a direct sum of a $2$-simplex and a $6$-gon. 
Using the data in Table~\ref{tab:VMNF918}, we obtain Table~\ref{d=4v=9f=18} which shows that the four cohomology rings are not isomorphic to each other. 

{\scriptsize
\begin{table}[H]
\centering
\begin{tabular}{|c|c|c|} \hline
ID & vertices of $P$ from $5$th & minimal nonfaces \\ \hline
27 & $(-1,-1,0,2), (0,0,-1,1), (0,0,1,-1),(0,0,-1,0), (0,0,0,-1)$ & 36, 38, 39, 47, 48, 49,  67, 69, 78, 125\\ 
\hline
46 & $(-1,-1,1,1), (0,0,-1,1), (0,0,1,-1),(0,0,-1,0), (0,0,0,-1)$ & 36, 38, 39, 47, 48, 49, 67, 69, 78, 125\\ 
\hline
119& $(-1,0,0,1), (1,0,0,-1), (-1,0,0,0),(0,-1,-1,1), (0,0,0,-1)$ & 15, 17, 19, 46, 47, 49, 56, 59, 67, 238 \\ 
\hline
122& $(-1,0,0,1), (1,0,0,-1), (-1,0,0,0),(0,0,0,-1), (0,-1,-1,0)$ & 15, 17, 18, 46, 47, 48, 56, 58, 67, 239\\ 
\hline
\end{tabular}
\medskip
\caption{$(V(P),F(P))=(9,18)$} \label{tab:VMNF918} 
\end{table}
\begin{table}[H]
\centering
\begin{tabular}{|c|c|c|c|} \hline
ID & $\mathcal{I}$ & s.v.e. $\ZZ/2$ & c.v.e. $\ZZ/3$ \\ \hline
27 & $y(y+u),u(y+u),v(z-u),z(2x-z-v),$ & $(y+u,z+v,u+v)$ & $(y+u,x)$ \\
&$u(2x+y-v),v(2x-z-v),yz,yv,zu,x^3$ & & \\ \hline
46 & $y(x-y-u),u(x-y-u),v(x+z-u),z(x-z-v),$ & $u+v$ & $x$ \\
&$u(x+y-v),v(x-z-v),yz,yv,zu,x^3$ & & \\ \hline
119 & $x(x+z),z(x+z),v(y-z),y(y-u+v),$ & $x+z$ & $(x+z,u)$ \\
&$z(x+u-v),v(y-u+v),xy,xv,yz,u^3$ & & \\ \hline
122 & $x(x+z),z(x+z),u(y-z),y(y+u),$ & $(x+z,y+u,z+u)$ & all \\
&$z(x-u),u(y+u),xy,xu,yz,v^3$ & & \\ \hline
\end{tabular}
\medskip
\caption{Ideals and invariants when $(V(P),F(P))=(9,18)$}
\label{d=4v=9f=18}
\end{table}
}


\subsubsection{$(V(P),F(P))=(9,20)$}

In this case, there are 17 smooth Fano $4$-polytopes as shown in Table~\ref{tab:VMNF920} and they are all combinatorially equivalent to a direct sum of two $1$-simplices and a $5$-gon.

{\scriptsize
\begin{table}[H]
\centering
\begin{tabular}{|c|c|c|} \hline
ID & vertices of $P$ from $5$th & minimal nonfaces \\ \hline
71 & $(-1,0,0,1), (0,-1,0,1), (0,0,-1,1),(0,0,1,-1), (0,0,-1,0)$ & 15, 26, 37, 39, 48, 49, 78\\ \hline
73 & $(-1,0,0,1), (0,-1,0,1), (0,0,-1,1),(0,0,1,-1), (0,0,0,-1)$ & 15, 26, 37, 39, 48, 49, 78\\ \hline
76 & $(-1,0,0,1), (0,-1,0,1), (0,0,1,-1),(0,0,-1,0), (0,0,0,-1)$ & 15, 26, 38, 39, 47, 49, 78\\ \hline
77 & $(-1,0,0,1), (0,-1,0,1), (0,1,-1,0),(0,1,0,-1), (0,-1,0,0)$ & 15, 26, 29, 37, 48, 49, 68\\ \hline
79 & $(-1,0,0,1), (0,-1,0,1), (0,1,-1,0),(1,0,0,-1), (-1,0,0,0)$ & 15, 19, 26, 37, 48, 49, 58\\ \hline
81 & $(-1,0,0,1), (0,-1,0,1), (0,1,-1,0),(1,0,0,-1), (0,0,0,-1)$ & 15, 19, 26, 37, 48, 49, 58\\ \hline
82 & $(-1,0,0,1), (0,-1,0,1), (0,1,-1,0),(0,-1,0,0), (0,0,0,-1)$ & 15, 26, 28, 37, 48, 49, 69\\ \hline
84 & $(-1,0,0,1), (0,-1,0,1), (0,1,0,-1),(0,-1,0,0), (0,0,-1,0)$ & 15, 26, 28, 39, 47, 48, 67\\ \hline
88 & $(-1,0,0,1), (0,-1,0,1), (0,1,0,-1),(0,-1,0,0), (0,1,-1,-1)$ & 15, 26, 28, 39, 47, 48, 67\\ \hline
90 & $(-1,0,0,1), (0,-1,0,1), (0,1,0,-1),(0,0,-1,0), (0,0,0,-1)$ & 15, 26, 29, 38, 47, 49, 67\\ \hline
91 & $(-1,0,0,1), (0,-1,0,1), (0,1,0,-1),(0,0,0,-1), (0,1,-1,-1)$ & 15, 26, 28, 39, 47, 48, 67\\ \hline
92 & $(-1,0,0,1), (0,-1,0,1), (0,1,0,-1),(0,0,0,-1), (0,0,-1,-1)$ & 15, 26, 28, 39, 47, 48, 67\\ \hline
102& $(-1,0,0,1), (0,-1,1,0), (0,1,-1,0),(0,-1,0,0), (0,0,0,-1)$ & 15, 26, 28, 37, 38, 49, 67 \\ \hline
103& $(-1,0,0,1), (0,-1,1,0), (0,1,-1,0),(0,-1,0,0), (0,1,-1,-1)$ & 15, 26, 28, 37,  38, 49, 67\\ \hline
107& $(-1,0,0,1), (0,-1,1,0), (0,0,-1,0),(0,0,0,-1), (0,0,-1,-1)$ & 15, 26, 37, 39, 48, 49, 78 \\ \hline
113& $(-1,0,0,1), (0,0,1,-1), (0,-1,0,0),(0,0,-1,0), (0,0,0,-1)$ & 15, 27, 38, 39, 46, 49, 68\\ \hline
120& $(-1,0,0,1), (1,0,0,-1), (-1,0,0,0),(0,-1,0,0), (0,0,-1,0)$ & 15, 17, 28, 39, 46, 47, 56\\ \hline
\end{tabular}
\medskip
\caption{$(V(P),F(P))=(9,20)$} \label{tab:VMNF920} 
\end{table}
}

We see that 
\begin{equation} \label{eq:diff920}
\begin{split}
&X_{73}\cong X_{76}\cong X_{92}, \quad X_{77}\cong X_{88},\quad X_{81}\cong X_{103},\\ 
&X_{82}\cong X_{91}\cong X_{107},\quad X_{90}\cong X_{113},
\end{split}
\end{equation}
see Subsection~\ref{subsec:diffeo}. Using the data in Table~\ref{tab:VMNF920}, we obtain the following table.

{\scriptsize
\begin{table}[H]
\centering
\begin{tabular}{|c|c|c|c|c|} \hline
ID & $\mathcal{I}$ & s.v.e. & s.v.e. $\ZZ/2$ & 4-v.e. $\ZZ/2$ \\ \hline
71 & \begin{tabular}{c}$x^2,y^2,z(z+v),v(z-u+v),$\\
$u(x+y-u),v(x+y-v),zu$\end{tabular} & $x,y$ & & \\ \hline
\begin{tabular}{c}73\\
(76,\ 92)\end{tabular} & \begin{tabular}{c}$x^2,y^2,z^2,v(z-u),$\\
$u(x+y-u-v),v(x+y-v),zu$\end{tabular} & $x,y,z$ & & \\ \hline
\begin{tabular}{c}77\\
(88) \end{tabular}& \begin{tabular}{c}$x^2,y(y-z+v),v(x+z-v),$\\
$z^2,u(x-u),v(x+y-u),yu$\end{tabular} & $x,z,x-2u$ & $(x,z)$ & $(x,z,u,v)$ \\ \hline
79 &\begin{tabular}{c} $x(x+v),v(v-y),y(y-z),$\\
$z^2,u(y-u),v(x+y-u),xu$\end{tabular} & $z,z-2y$ & & \\ \hline
\begin{tabular}{c}81\\
(103)\end{tabular} & \begin{tabular}{c}$x^2,v(x-u),y(y-z),z^2,$\\
$u(y-u-v),v(y-v),xu$ \end{tabular}& $x,z,z-2y$ & $(x,z)$ & $(x,y,z,v)$\\ \hline
\begin{tabular}{c}82\\
(91,\ 107)\end{tabular} & \begin{tabular}{c}$x^2,y(y-z+u),u(y-z+u),$\\
$z^2,u(x+y-v),v(x-v),yv$ \end{tabular}& \begin{tabular}{c}$x,z,x-2v,$\\
$z-2y-2u$\end{tabular} & $(x,z)$ & $(x,y+u,z,v)$ \\ \hline
84 & \begin{tabular}{c}$x^2,y(y+u),u(x-u),v^2,$\\
$u(x+y-z),z(x-z),yz$ \end{tabular}& \begin{tabular}{c}$x,v,x-2u,$\\
$x-2z$\end{tabular} & $(x,v)$ & all \\ \hline
\begin{tabular}{c}90\\
(113)\end{tabular} & \begin{tabular}{c}$x^2,y^2,v(y-z),u^2,$\\
$v(x-v),z(x-z-v),yz$ \end{tabular}& \begin{tabular}{c}$x,y,u,$\\
$x-2v$\end{tabular} & & \\ \hline
102 & \begin{tabular}{c}$x^2,y(y+u),u^2,z^2,$\\
$u(y-z),v(x-v),yz$\end{tabular} & $\infty$ & $(x,z,u)$ & \\ \hline
120 & \begin{tabular}{c}$x(x+z),z^2,u^2,v^2,$\\
$y^2,z(x-y),xy$\end{tabular} & $\infty$& $(y,z,u,v)$ & \\ \hline
\end{tabular}
\medskip
\caption{Ideals and invariants when $(V(P), F(P))=(9,20)$}
\label{d=4v=9f=20}
\end{table}
}

Table~\ref{d=4v=9f=20} shows that the ten cohomology rings in the table are distinguished by the invariants in the table except for $77$ and $81$. 
The s.v.e. over $\ZZ/2$ of $H^*(X_{77})$ and $ H^*(X_{81})$ are both $(x,z)$.  Therefore, if there is an isomorphism $H^*(X_{77})\to H^*(X_{81})$, 
then it induces an isomorphism $(H^*(X_{77})\otimes \ZZ/2)/(x,z)\to (H^*(X_{81})\otimes \ZZ/2)/(x,z)$. However, Table~\ref{d=4v=9f=20-1} shows that this does not occur because the degree sequences of the ideals are different, so $H^*(X_{77})\ncong H^*(X_{81})$. 

{\scriptsize
\begin{table}[H]
\centering
\begin{tabular}{|c|c|c|} \hline
ID & $\mathcal{I}\otimes \ZZ/2+(x,z)$  \\ \hline
77(88) & $y(y+v),v^2,u^2,v(y+u),yu$ \\ \hline
81(103) & $uv,y^2,u(y+u),v(y+v)$ \\ \hline
\end{tabular}
\medskip
\caption{Distinguishment between $H^*(X_{77})$ and $H^*(X_{81})$}\label{d=4v=9f=20-1}
\end{table}
}


\subsubsection{$(V(P),F(P))=(9,21)$}

In this case, there are four smooth Fano $4$-polytopes and there are two combinatorial types among them. Indeed, the combinatorial type of ID $52$ is different from the others as is seen from Table~\ref{tab:VMNF921}.

{\scriptsize
\begin{table}[H]
\begin{tabular}{|c|c|c|} \hline
ID & vertices of $P$ from $5$th & minimal nonfaces \\ \hline
51 & $(-1,-1,1,1), (0,1,-1,0), (0,1,0,-1),(0,-1,0,0), (0,0,-1,0)$ & 28, 29, 36, 39, 47, 68, 125, 156, 157, 159, 348\\ \hline
52 & $(-1,-1,1,1), (0,1,-1,0), (1,0,0,-1),(0,0,-1,0), (0,0,0,-1)$ & 19, 28, 36, 38, 47, 49, 125, 156, 257, 567\\ 
\hline
56 & $(-1,-1,1,1), (0,1,-1,0), (0,-1,0,0),(0,0,-1,0), (0,0,0,-1)$ & 27, 28, 36, 38, 49, 67, 125, 156, 158, 159, 347\\ \hline
89 & $(-1,0,0,1), (0,-1,0,1), (0,1,0,-1),(0,-1,0,0), (1,0,-1,-1)$ & 15, 26, 28, 47, 48, 67, 178, 239, 349, 359, 369\\ \hline
\end{tabular}
\medskip
\caption{$(V(P),F(P))=(9,21)$} \label{tab:VMNF921} 
\end{table}
}

%

Using the data in Table~\ref{tab:VMNF921}, we obtain Table~\ref{d=4v=9f=21}. 
It shows that $H^*(X_{52})$ can be distinguished from the other three cohomology rings by the degree sequences of the ideals.
On the other hand, Table~\ref{d=4v=9f=21/2Z} shows that the three cohomology rings are not isomorphic to each other. 

{\scriptsize
\begin{table}[H]
\centering
\begin{tabular}{|c|c|} \hline
ID & $\mathcal{I}$ \\ \hline
51 & $u(x-z+u),v(x-y-z+u),y(x-y-v),v(x-y-v),$ \\
& $z(x-z),yu,x^2(x+u),x^2y,x^2z,x^2v,u(x-v)(x-z)$ \\ \hline
52 & $v(x-z),u(x-y),y(x-y-u),u^2,z(x-z-v),$ \\
& $v^2,x^3,x^2y,x^2z,xyz$ \\ \hline
56 & $z(x+z),u(u+z),y(x-y-u),u(x-y-u),v(x-v),$ \\
& $yz,x^2(x+z),x^2y,x^2u,x^2v,z(x-u)(x-v)$ \\ \hline
89 & $x(x-v),y(y+u),u(y-z+u),z(x-z-v),u(x-u-v),$ \\
& $yz,zu(x-v),v^2(z-u),v^2(z+v),v^2x,v^2y$ \\ \hline
\end{tabular}
\medskip
\caption{Ideals when $(V(P),F(P))=(9,21)$}\label{d=4v=9f=21}
\end{table}
}


{\scriptsize
\begin{table}[H]
	\centering
	\tiny
	\begin{tabular}{|c|c|c|c|} \hline
		ID & s.v.e. $\ZZ/2\ZZ$ & c.v.e. $\ZZ/3\ZZ$ & c.v.e. $\ZZ/2\ZZ$ \\ \hline
		51 & $\emptyset$ & $(z,y+v)$ & $x+u,x+u+v,z,y+z,y+v,y+z+v$ \\ \hline
                52 & $(u,v)$ &  &  \\ \hline
		56 & $\emptyset$ & $(x+z,y+u,v)$ & \\ \hline 
		89 & $\emptyset$ & $(x,y+u)$ & $x,x+z+u,y+u,z+v,u+v$ \\ \hline
	\end{tabular}
	\medskip
	\caption{$(V(P),F(P))=(9,21)$} 
	\label{d=4v=9f=21/2Z}
\end{table}
}

\bigskip

\subsection{Diffeomorphism} \label{subsec:diffeo}

In the previous subsections, we claimed some diffeomorphisms among toric Fano $4$-folds, i.e. \eqref{eq:diff69}, \eqref{eq:diff712}, \eqref{eq:diff815}, \eqref{eq:diff816-1}, \eqref{eq:diff816-2}, \eqref{eq:diff920}. In this subsection, we establish them by showing that the condition in Lemma~\ref{diffeolemma} is satisfied in each case. 

In the following, we express the vertices of a smooth Fano polytope $P_q$ with ID number $q$ in terms of a matrix where each row shows a vertex of $P_q$ and the numbers written on the left side of a matrix are the numbers of the vertices. For instance, the vertices of $P_{70}$ are arranged as $(1,2,3,4,5,6)$ while the vertices of $P_{141}$ are arranged as $(1,2,3,6,5,4)$ in their matrices. The correspondence $(1,2,3,4,5,6)\to (1,2,3,6,5,4)$ gives a bijection between the vertices of $P_{70}$ and $P_{141}$, which preserves the combinatorial structures of $P_{70}$ and $P_{141}$. The multiplication by a $4\times 4$ unimodular matrix from the right shows a unimodular transformation. In each case, we will see that the resulting vectors of the corresponding vertices agree up to sign, so the condition in Lemma~\ref{diffeolemma} is satisfied.

$\bullet$ 70 and 141

{\tiny 
\begin{table}[H]
\centering
\begin{tabular}{ccccccccc}
\begin{tabular}{c}
\\
$1$ \\
$2$ \\
$3$ \\
$4$ \\
$5$ \\
$6$ \\
\end{tabular}
\begin{tabular}{c}
\textbf{70} \\
$\begin{pmatrix}
1 & 0 & 0 & 0 \\
0 & 1 & 0 & 0 \\
0 & 0 & 1 & 0 \\
0 & 0 & 0 & 1 \\
-1 & -1 & 1 & 1 \\
0 & 0 & -1 & -1 \\
\end{pmatrix}$
\end{tabular}
& & & 
\begin{tabular}{c}
\\
$1$ \\
$2$ \\
$3$ \\
$6$ \\
$5$ \\
$4$ \\
\end{tabular}
\begin{tabular}{l}
\hspace{4.5em}\textbf{141} \\
$\begin{pmatrix}
1 & 0 & 0 & 0 \\
0 & 1 & 0 & 0 \\
0 & 0 & 1 & 0 \\
0 & 0 & -1 & -1 \\
-1 & -1 & 0 & 1 \\
0 & 0 & 0 & 1 \\
\end{pmatrix} 
\begin{pmatrix}
1 & 0 & 0 & 0\\
0 & 1 & 0 & 0\\
0 & 0 & -1 & 0\\
0 & 0 & 1 & 1\end{pmatrix}
=\begin{pmatrix}
1 & 0 & 0 & 0 \\
0 & 1 & 0 & 0 \\
0 & 0 & -1 & 0 \\
0 & 0 & 0 & -1 \\
-1 & -1 & 1 & 1 \\
0 & 0 & 1 & 1 \\
\end{pmatrix} 
$
\end{tabular}
\end{tabular}
\end{table}
}

$\bullet$ 30 and 43
{\tiny 
\begin{table}[H]
\centering
\begin{tabular}{ccccccccc}
\begin{tabular}{c}
\\
$1$ \\
$2$ \\
$3$ \\
$4$ \\
$5$ \\
$6$ \\
$7$ \\
\end{tabular}
\begin{tabular}{c}
\textbf{30} \\
$\begin{pmatrix}
1 & 0 & 0 & 0 \\
0 & 1 & 0 & 0 \\
0 & 0 & 1 & 0 \\
0 & 0 & 0 & 1 \\
-1 & -1 & 0 & 2 \\
0 & 0 & -1 & 1 \\
0 & 0 & 0 & -1 \\
\end{pmatrix}$
\end{tabular}
& & & 
\begin{tabular}{c}
\\
$1$ \\
$2$ \\
$3$ \\
$4$ \\
$5$ \\
$7$ \\
$6$ \\
\end{tabular}
\begin{tabular}{l}
\hspace{4.75em}\textbf{43} \\
$\begin{pmatrix}
1 & 0 & 0 & 0 \\
0 & 1 & 0 & 0 \\
0 & 0 & 1 & 0 \\
0 & 0 & 0 & 1 \\
-1 & -1 & 0 & 2 \\
0 & 0 & -1 & -1 \\
0 & 0 & 0 & -1 \\
\end{pmatrix} 
\begin{pmatrix}
-1 & 0 & 0 & 0\\
0 & -1 & 0 & 0\\
0 & 0 & 1 & 0\\
0 & 0 & 0 & -1\end{pmatrix}
=\begin{pmatrix}
-1 & 0 & 0 & 0 \\
0 & -1 & 0 & 0 \\
0 & 0 & 1 & 0 \\
0 & 0 & 0 & -1 \\
1 & 1 & 0 & -2 \\
0 & 0 & -1 & 1 \\
0 & 0 & 0 & 1 \\
\end{pmatrix} 
$
\end{tabular}
\end{tabular}
\end{table}
}

$\bullet$ 68 and 134
{\tiny 
\begin{table}[H]
\centering
\begin{tabular}{ccccccccc}
\begin{tabular}{c}
\\
$1$ \\
$2$ \\
$3$ \\
$4$ \\
$5$ \\
$6$ \\
$7$ \\
\end{tabular}
\begin{tabular}{c}
\textbf{68} \\
$\begin{pmatrix}
1 & 0 & 0 & 0 \\
0 & 1 & 0 & 0 \\
0 & 0 & 1 & 0 \\
0 & 0 & 0 & 1 \\
-1 & -1 & 1 & 1 \\
0 & 0 & -1 & 0 \\
0 & 0 & -1 & -1 \\
\end{pmatrix}$
\end{tabular}
& & & 
\begin{tabular}{c}
\\
$2$ \\
$3$ \\
$6$ \\
$5$ \\
$7$ \\
$4$ \\
$1$ \\
\end{tabular}
\begin{tabular}{l}
\hspace{4.5em}\textbf{134} \\
$\begin{pmatrix}
0 & 1 & 0 & 0 \\
0 & 0 & 1 & 0 \\
0 & 0 & 0 & -1 \\
-1 & 0 & 0 & 1 \\
1 & -1 & -1 & 0 \\
0 & 0 & 0 & 1 \\
1 & 0 & 0 & 0 \\
\end{pmatrix} 
\begin{pmatrix}
0 & 0 & 1 & 1\\
1 & 0 & 0 & 0\\
0 & 1 & 0 & 0\\
0 & 0 & 1 & 0\end{pmatrix}
=\begin{pmatrix}
1 & 0 & 0 & 0 \\
0 & 1 & 0 & 0 \\
0 & 0 & -1 & 0 \\
0 & 0 & 0 & -1 \\
-1 & -1 & 1 & 1 \\
0 & 0 & 1 & 0 \\
0 & 0 & 1 & 1 \\
\end{pmatrix} 
$
\end{tabular}
\end{tabular}
\end{table}
}

$\bullet$ 129 and 136
{\tiny 
\begin{table}[H]
\centering
\begin{tabular}{ccccccccc}
\begin{tabular}{c}
\\
$1$ \\
$2$ \\
$3$ \\
$4$ \\
$5$ \\
$6$ \\
$7$ \\
\end{tabular}
\begin{tabular}{c}
\textbf{129} \\
$\begin{pmatrix}
1 & 0 & 0 & 0 \\
0 & 1 & 0 & 0 \\
0 & 0 & 1 & 0 \\
0 & 0 & 0 & 1 \\
-1 & 0 & 0 & 1 \\
0 & -1 & -1 & 1 \\
0 & 0 & 0 & -1 \\
\end{pmatrix}$
\end{tabular}
& & & 
\begin{tabular}{c}
\\
$1$ \\
$2$ \\
$3$ \\
$4$ \\
$5$ \\
$7$ \\
$6$ \\
\end{tabular}
\begin{tabular}{l}
\hspace{4.5em}\textbf{136} \\
$\begin{pmatrix}
1 & 0 & 0 & 0 \\
0 & 1 & 0 & 0 \\
0 & 0 & 1 & 0 \\
0 & 0 & 0 & 1 \\
-1 & 0 & 0 & 1 \\
0 & -1 & -1 & -1 \\
0 & 0 & 0 & -1 \\
\end{pmatrix} 
\begin{pmatrix}
-1 & 0 & 0 & 0\\
0 & 1 & 0 & 0\\
0 & 0 & 1 & 0\\
0 & 0 & 0 & -1\end{pmatrix}
=\begin{pmatrix}
-1 & 0 & 0 & 0 \\
0 & 1 & 0 & 0 \\
0 & 0 & 1 & 0 \\
0 & 0 & 0 & -1 \\
1 & 0 & 0 & -1 \\
0 & -1 & -1 & 1 \\
0 & 0 & 0 & 1 \\
\end{pmatrix} 
$
\end{tabular}
\end{tabular}
\end{table}
}

$\bullet$ 28 and 32
{\tiny 
\begin{table}[H]
\centering
\begin{tabular}{ccccccccc}
\begin{tabular}{c}
\\
$1$ \\
$2$ \\
$3$ \\
$4$ \\
$5$ \\
$6$ \\
$7$ \\
$8$ \\
\end{tabular}
\begin{tabular}{c}
\textbf{28} \\
$\begin{pmatrix}
1 & 0 & 0 & 0 \\
0 & 1 & 0 & 0 \\
0 & 0 & 1 & 0 \\
0 & 0 & 0 & 1 \\
-1 & -1 & 0 & 2 \\
0 & 0 & -1 & 1 \\
0 & 0 & 1 & -1 \\
0 & 0 & 0 & -1 \\
\end{pmatrix}$
\end{tabular}
& & & 
\begin{tabular}{c}
\\
$1$ \\
$2$ \\
$6$ \\
$8$ \\
$5$ \\
$7$ \\
$3$ \\
$4$ \\
\end{tabular}
\begin{tabular}{l}
\hspace{4.75em}\textbf{32} \\
$\begin{pmatrix}
1 & 0 & 0 & 0 \\
0 & 1 & 0 & 0 \\
0 & 0 & 1 & -1 \\
0 & 0 & 0 & -1 \\
-1 & -1 & 0 & 2 \\
0 & 0 & -1 & 0 \\
0 & 0 & 1 & 0 \\
0 & 0 & 0 & 1 \\
\end{pmatrix} 
\begin{pmatrix}
1 & 0 & 0 & 0\\
0 & 1 & 0 & 0\\
0 & 0 & -1 & 1\\
0 & 0 & 0 & 1\end{pmatrix}
=\begin{pmatrix}
1 & 0 & 0 & 0 \\
0 & 1 & 0 & 0 \\
0 & 0 & -1 & 0 \\
0 & 0 & 0 & -1 \\
-1 & -1 & 0 & 2 \\
0 & 0 & 1 & -1 \\
0 & 0 & -1 & 1 \\
0 & 0 & 0 & 1 \\
\end{pmatrix} 
$
\end{tabular}
\end{tabular}
\end{table}
}

$\bullet$ 67 and 118
{\tiny 
\begin{table}[H]
\centering
\begin{tabular}{ccccccccc}
\begin{tabular}{c}
\\
$1$ \\
$2$ \\
$3$ \\
$4$ \\
$5$ \\
$6$ \\
$7$ \\
$8$ \\
\end{tabular}
\begin{tabular}{c}
\textbf{67} \\
$\begin{pmatrix}
1 & 0 & 0 & 0 \\
0 & 1 & 0 & 0 \\
0 & 0 & 1 & 0 \\
0 & 0 & 0 & 1 \\
-1 & -1 & 1 & 1 \\
0 & 0 & -1 & 0 \\
0 & 0 & 0 & -1 \\
0 & 0 & -1 & -1 \\
\end{pmatrix}$
\end{tabular}
& & & 
\begin{tabular}{c}
\\
$2$ \\
$3$ \\
$7$ \\
$6$ \\
$8$ \\
$1$ \\
$5$ \\
$4$ \\
\end{tabular}
\begin{tabular}{l}
\hspace{4.5em}\textbf{118} \\
$\begin{pmatrix}
0 & 1 & 0 & 0 \\
0 & 0 & 1 & 0 \\
-1 & 0 & 0 & 0 \\
1 & 0 & 0 & -1 \\
0 & -1 & -1 & 1 \\
1 & 0 & 0 & 0 \\
-1 & 0 & 0 & 1 \\
0 & 0 & 0 & 1 \\
\end{pmatrix} 
\begin{pmatrix}
0 & 0 & 1 & 0\\
1 & 0 & 0 & 0\\
0 & 1 & 0 & 0\\
0 & 0 & 1 & 1\end{pmatrix}
=\begin{pmatrix}
1 & 0 & 0 & 0 \\
0 & 1 & 0 & 0 \\
0 & 0 & -1 & 0 \\
0 & 0 & 0 & -1 \\
-1 & -1 & 1 & 1 \\
0 & 0 & 1 & 0 \\
0 & 0 & 0 & 1 \\
0 & 0 & 1 & 1 \\
\end{pmatrix} 
$
\end{tabular}
\end{tabular}
\end{table}
}

$\bullet$ 123 and 125
{\tiny 
\begin{table}[H]
\centering
\begin{tabular}{ccccccccc}
\begin{tabular}{c}
\\
$1$ \\
$2$ \\
$3$ \\
$4$ \\
$5$ \\
$6$ \\
$7$ \\
$8$ \\
\end{tabular}
\begin{tabular}{c}
\textbf{123} \\
$\begin{pmatrix}
1 & 0 & 0 & 0 \\
0 & 1 & 0 & 0 \\
0 & 0 & 1 & 0 \\
0 & 0 & 0 & 1 \\
-1 & 0 & 0 & 1 \\
1 & 0 & 0 & -1 \\
-1 & 0 & 0 & 0 \\
1 & -1 & -1 & 0 \\
\end{pmatrix}$
\end{tabular}
& & & 
\begin{tabular}{c}
\\
$5$ \\
$2$ \\
$3$ \\
$4$ \\
$1$ \\
$7$ \\
$6$ \\
$8$ \\
\end{tabular}
\begin{tabular}{l}
\hspace{4.5em}\textbf{125} \\
$\begin{pmatrix}
-1 & 0 & 0 & 1 \\
0 & 1 & 0 & 0 \\
0 & 0 & 1 & 0 \\
0 & 0 & 0 & 1 \\
1 & 0 & 0 & 0 \\
-1 & 0 & 0 & 0 \\
1 & 0 & 0 & -1 \\
1 & -1 & -1 & -1 \\
\end{pmatrix} 
\begin{pmatrix}
1 & 0 & 0 & -1\\
0 & 1 & 0 & 0\\
0 & 0 & 1 & 0\\
0 & 0 & 0 & -1\end{pmatrix}
=\begin{pmatrix}
-1 & 0 & 0 & 0 \\
0 & 1 & 0 & 0 \\
0 & 0 & 1 & 0 \\
0 & 0 & 0 & -1 \\
1 & 0 & 0 & -1 \\
-1 & 0 & 0 & 1 \\
1 & 0 & 0 & 0 \\
1 & -1 & -1 & 0 \\
\end{pmatrix} 
$
\end{tabular}
\end{tabular}
\end{table}
}

$\bullet$ 74 and 96
{\tiny 
\begin{table}[H]
\centering
\begin{tabular}{ccccccccc}
\begin{tabular}{c}
\\
$1$ \\
$2$ \\
$3$ \\
$4$ \\
$5$ \\
$6$ \\
$7$ \\
$8$ \\
\end{tabular}
\begin{tabular}{c}
\textbf{74} \\
$\begin{pmatrix}
1 & 0 & 0 & 0 \\
0 & 1 & 0 & 0 \\
0 & 0 & 1 & 0 \\
0 & 0 & 0 & 1 \\
-1 & 0 & 0 & 1 \\
0 & -1 & 0 & 1 \\
0 & 0 & -1 & 1 \\
0 & 0 & 0 & -1 \\
\end{pmatrix}$
\end{tabular}
& & & 
\begin{tabular}{c}
\\
$1$ \\
$2$ \\
$3$ \\
$4$ \\
$5$ \\
$6$ \\
$8$ \\
$7$ \\
\end{tabular}
\begin{tabular}{l}
\hspace{4.75em}\textbf{96} \\
$\begin{pmatrix}
1 & 0 & 0 & 0 \\
0 & 1 & 0 & 0 \\
0 & 0 & 1 & 0 \\
0 & 0 & 0 & 1 \\
-1 & 0 & 0 & 1 \\
0 & -1 & 0 & 1 \\
0 & 0 & -1 & -1 \\
0 & 0 & 0 & -1 \\
\end{pmatrix} 
\begin{pmatrix}
-1 & 0 & 0 & 0\\
0 & -1 & 0 & 0\\
0 & 0 & 1 & 0\\
0 & 0 & 0 & -1\end{pmatrix}
=\begin{pmatrix}
-1 & 0 & 0 & 0 \\
0 & -1 & 0 & 0 \\
0 & 0 & 1 & 0 \\
0 & 0 & 0 & -1 \\
1 & 0 & 0 & -1 \\
0 & 1 & 0 & -1 \\
0 & 0 & -1 & 1 \\
0 & 0 & 0 & 1 \\
\end{pmatrix} 
$
\end{tabular}
\end{tabular}
\end{table}
}

$\bullet$ 83 and 108
{\tiny 
\begin{table}[H]
\centering
\begin{tabular}{ccccccccc}
\begin{tabular}{c}
\\
$1$ \\
$2$ \\
$3$ \\
$4$ \\
$5$ \\
$6$ \\
$7$ \\
$8$ \\
\end{tabular}
\begin{tabular}{c}
\textbf{83} \\
$\begin{pmatrix}
1 & 0 & 0 & 0 \\
0 & 1 & 0 & 0 \\
0 & 0 & 1 & 0 \\
0 & 0 & 0 & 1 \\
-1 & 0 & 0 & 1 \\
0 & -1 & 0 & 1 \\
0 & 1 & -1 & 0 \\
0 & 0 & 0 & -1 \\
\end{pmatrix}$
\end{tabular}
& & & 
\begin{tabular}{c}
\\
$2$ \\
$4$ \\
$1$ \\
$3$ \\
$6$ \\
$8$ \\
$5$ \\
$7$ \\
\end{tabular}
\begin{tabular}{l}
\hspace{4.5em}\textbf{108} \\
$\begin{pmatrix}
0 & 1 & 0 & 0 \\
0 & 0 & 0 & 1 \\
1 & 0 & 0 & 0 \\
0 & 0 & 1 & 0 \\
0 & -1 & 1 & 0 \\
0 & 0 & -1 & -1 \\
-1 & 0 & 0 & 1 \\
0 & 0 & -1 & 0 \\
\end{pmatrix} 
\begin{pmatrix}
0 & 0 & 1 & 0\\
-1 & 0 & 0 & 0\\
0 & 0 & 0 & -1\\
0 & 1 & 0 & 0\end{pmatrix}
=\begin{pmatrix}
-1 & 0 & 0 & 0 \\
0 & 1 & 0 & 0 \\
0 & 0 & 1 & 0 \\
0 & 0 & 0 & -1 \\
1 & 0 & 0 & -1 \\
0 & -1 & 0 & 1 \\
0 & 1 & -1 & 0 \\
0 & 0 & 0 & 1 \\
\end{pmatrix} 
$
\end{tabular}
\end{tabular}
\end{table}
}

$\bullet$ 95 and 131
{\tiny 
\begin{table}[H]
\centering
\begin{tabular}{ccccccccc}
\begin{tabular}{c}
\\
$1$ \\
$2$ \\
$3$ \\
$4$ \\
$5$ \\
$6$ \\
$7$ \\
$8$ \\
\end{tabular}
\begin{tabular}{c}
\textbf{95} \\
$\begin{pmatrix}
1 & 0 & 0 & 0 \\
0 & 1 & 0 & 0 \\
0 & 0 & 1 & 0 \\
0 & 0 & 0 & 1 \\
-1 & 0 & 0 & 1 \\
0 & -1 & 0 & 1 \\
0 & 0 & -1 & 0 \\
0 & 0 & 0 & -1 \\
\end{pmatrix}$
\end{tabular}
& & & 
\begin{tabular}{c}
\\
$1$ \\
$3$ \\
$2$ \\
$4$ \\
$5$ \\
$8$ \\
$6$ \\
$7$ \\
\end{tabular}
\begin{tabular}{l}
\hspace{4.5em}\textbf{131} \\
$\begin{pmatrix}
1 & 0 & 0 & 0 \\
0 & 0 & 1 & 0 \\
0 & 1 & 0 & 0 \\
0 & 0 & 0 & 1 \\
-1 & 0 & 0 & 1 \\
0 & 0 & -1 & -1 \\
0 & -1 & 0 & 0 \\
0 & 0 & 0 & -1 \\
\end{pmatrix} 
\begin{pmatrix}
-1 & 0 & 0 & 0\\
0 & 0 & 1 & 0\\
0 & 1 & 0 & 0\\
0 & 0 & 0 & -1\end{pmatrix}
=\begin{pmatrix}
-1 & 0 & 0 & 0 \\
0 & 1 & 0 & 0 \\
0 & 0 & 1 & 0 \\
0 & 0 & 0 & -1 \\
1 & 0 & 0 & -1 \\
0 & -1 & 0 & 1 \\
0 & 0 & -1 & 0 \\
0 & 0 & 0 & 1 \\
\end{pmatrix} 
$
\end{tabular}
\end{tabular}
\end{table}
}

$\bullet$ 29 and 39
{\tiny 
\begin{table}[H]
\centering
\begin{tabular}{ccccccccc}
\begin{tabular}{c}
\\
$1$ \\
$2$ \\
$3$ \\
$4$ \\
$5$ \\
$6$ \\
$7$ \\
$8$ \\
\end{tabular}
\begin{tabular}{c}
\textbf{29} \\
$\begin{pmatrix}
1 & 0 & 0 & 0 \\
0 & 1 & 0 & 0 \\
0 & 0 & 1 & 0 \\
0 & 0 & 0 & 1 \\
-1 & -1 & 0 & 2 \\
0 & 0 & -1 & 1 \\
0 & 1 & 0 & -1 \\
0 & 0 & 0 & -1 \\
\end{pmatrix}$
\end{tabular}
& & & 
\begin{tabular}{c}
\\
$1$ \\
$3$ \\
$2$ \\
$4$ \\
$5$ \\
$8$ \\
$6$ \\
$7$ \\
\end{tabular}
\begin{tabular}{l}
\hspace{4.75em}\textbf{39} \\
$\begin{pmatrix}
1 & 0 & 0 & 0 \\
0 & 1 & 0 & 0 \\
0 & 0 & 1 & 0 \\
0 & 0 & 0 & 1 \\
-1 & -1 & 0 & 2 \\
0 & 0 & -1 & -1 \\
0 & 1 & 0 & -1 \\
0 & 0 & 0 & -1 \\
\end{pmatrix} 
\begin{pmatrix}
-1 & 0 & 0 & 0\\
0 & -1 & 0 & 0\\
0 & 0 & 1 & 0\\
0 & 0 & 0 & -1\end{pmatrix}
=\begin{pmatrix}
-1 & 0 & 0 & 0 \\
0 & -1 & 0 & 0 \\
0 & 0 & 1 & 0 \\
0 & 0 & 0 & -1 \\
1 & 1 & 0 & -2 \\
0 & 0 & -1 & 1 \\
0 & -1 & 0 & 1 \\
0 & 0 & 0 & 1 \\
\end{pmatrix} 
$
\end{tabular}
\end{tabular}
\end{table}
}

$\bullet$ 111 and 116
{\tiny 
\begin{table}[H]
\centering
\begin{tabular}{ccccccccc}
\begin{tabular}{c}
\\
$1$ \\
$2$ \\
$3$ \\
$4$ \\
$5$ \\
$6$ \\
$7$ \\
$8$ \\
\end{tabular}
\begin{tabular}{c}
\textbf{111} \\
$\begin{pmatrix}
1 & 0 & 0 & 0 \\
0 & 1 & 0 & 0 \\
0 & 0 & 1 & 0 \\
0 & 0 & 0 & 1 \\
-1 & 0 & 0 & 1 \\
0 & 0 & 1 & -1 \\
0 & -1 & -1 & 1 \\
0 & 0 & 0 & -1 \\
\end{pmatrix}$
\end{tabular}
& & & 
\begin{tabular}{c}
\\
$1$ \\
$2$ \\
$6$ \\
$7$ \\
$5$ \\
$3$ \\
$8$ \\
$4$ \\
\end{tabular}
\begin{tabular}{l}
\hspace{4.5em}\textbf{116} \\
$\begin{pmatrix}
1 & 0 & 0 & 0 \\
0 & 1 & 0 & 0 \\
0 & 0 & 1 & -1 \\
0 & 0 & 0 & -1 \\
-1 & 0 & 0 & 1 \\
0 & 0 & 1 & 0 \\
0 & -1 & -1 & 0 \\
0 & 0 & 0 & 1 \\
\end{pmatrix} 
\begin{pmatrix}
1 & 0 & 0 & 0\\
0 & -1 & 0 & 0\\
0 & 0 & -1 & 1\\
0 & 0 & 0 & 1\end{pmatrix}
=\begin{pmatrix}
1 & 0 & 0 & 0 \\
0 & -1 & 0 & 0 \\
0 & 0 & -1 & 0 \\
0 & 0 & 0 & -1 \\
-1 & 0 & 0 & 1 \\
0 & 0 & -1 & 1 \\
0 & 1 & 1 & -1 \\
0 & 0 & 0 & 1 \\
\end{pmatrix} 
$
\end{tabular}
\end{tabular}
\end{table}
}

$\bullet$ 73 and 76 and 92
{\tiny 
\begin{table}[H]
\centering
\begin{tabular}{ccccccccc}
\begin{tabular}{c}
\\
$1$ \\
$2$ \\
$3$ \\
$4$ \\
$5$ \\
$6$ \\
$7$ \\
$8$ \\
$9$ \\
\end{tabular}
\begin{tabular}{c}
\textbf{73} \\
$\begin{pmatrix}
1 & 0 & 0 & 0 \\
0 & 1 & 0 & 0 \\
0 & 0 & 1 & 0 \\
0 & 0 & 0 & 1 \\
-1 & 0 & 0 & 1 \\
0 & -1 & 0 & 1 \\
0 & 0 & -1 & 1 \\
0 & 0 & 1 & -1 \\
0 & 0 & 0 & -1 \\
\end{pmatrix}$
\end{tabular}
& & & 
\begin{tabular}{c}
\\
$1$ \\
$2$ \\
$7$ \\
$9$ \\
$5$ \\
$6$ \\
$8$ \\
$3$ \\
$4$ \\
\end{tabular}
\begin{tabular}{l}
\hspace{4.75em}\textbf{76} \\
$\begin{pmatrix}
1 & 0 & 0 & 0 \\
0 & 1 & 0 & 0 \\
0 & 0 & 1 & -1 \\
0 & 0 & 0 & -1 \\
-1 & 0 & 0 & 1 \\
0 & -1 & 0 & 1 \\
0 & 0 & -1 & 0 \\
0 & 0 & 1 & 0 \\
0 & 0 & 0 & 1 \\
\end{pmatrix} 
\begin{pmatrix}
1 & 0 & 0 & 0\\
0 & 1 & 0 & 0\\
0 & 0 & -1 & 1\\
0 & 0 & 0 & 1\end{pmatrix}
=\begin{pmatrix}
1 & 0 & 0 & 0 \\
0 & 1 & 0 & 0 \\
0 & 0 & -1 & 0 \\
0 & 0 & 0 & -1 \\
-1 & 0 & 0 & 1 \\
0 & -1 & 0 & 1 \\
0 & 0 & 1 & -1 \\
0 & 0 & -1 & 1 \\
0 & 0 & 0 & 1 \\
\end{pmatrix} 
$
\end{tabular}
\end{tabular}
\end{table}
}

{\tiny 
\begin{table}[H]
\centering
\begin{tabular}{ccccccccc}
& & & 
\begin{tabular}{c}
\\
$1$ \\
$3$ \\
$2$ \\
$4$ \\
$5$ \\
$9$ \\
$6$ \\
$7$ \\
$8$ \\
\end{tabular}
\begin{tabular}{l}
\hspace{4.75em}\textbf{92} \\
$\begin{pmatrix}
1 & 0 & 0 & 0 \\
0 & 0 & 1 & 0 \\
0 & 1 & 0 & 0 \\
0 & 0 & 0 & 1 \\
-1 & 0 & 0 & 1 \\
0 & 0 & -1 & -1 \\
0 & -1 & 0 & 1 \\
0 & 1 & 0 & -1 \\
0 & 0 & 0 & -1 \\
\end{pmatrix} 
\begin{pmatrix}
-1 & 0 & 0 & 0\\
0 & 0 & -1 & 0\\
0 & 1 & 0 & 0\\
0 & 0 & 0 & -1\end{pmatrix}
=\begin{pmatrix}
-1 & 0 & 0 & 0 \\
0 & 1 & 0 & 0 \\
0 & 0 & -1 & 0 \\
0 & 0 & 0 & -1 \\
1 & 0 & 0 & -1 \\
0 & -1 & 0 & 1 \\
0 & 0 & 1 & -1 \\
0 & 0 & -1 & 1 \\
0 & 0 & 0 & 1 \\
\end{pmatrix} 
$
\end{tabular}
\end{tabular}
\end{table}
}

$\bullet$ 77 and 88
{\tiny 
\begin{table}[H]
\centering
\begin{tabular}{ccccccccc}
\begin{tabular}{c}
\\
$1$ \\
$2$ \\
$3$ \\
$4$ \\
$5$ \\
$6$ \\
$7$ \\
$8$ \\
$9$ \\
\end{tabular}
\begin{tabular}{c}
\textbf{77} \\
$\begin{pmatrix}
1 & 0 & 0 & 0 \\
0 & 1 & 0 & 0 \\
0 & 0 & 1 & 0 \\
0 & 0 & 0 & 1 \\
-1 & 0 & 0 & 1 \\
0 & -1 & 0 & 1 \\
0 & 1 & -1 & 0 \\
0 & 1 & 0 & -1 \\
0 & -1 & 0 & 0 \\
\end{pmatrix}$
\end{tabular}
& & & 
\begin{tabular}{c}
\\
$1$ \\
$6$ \\
$3$ \\
$4$ \\
$5$ \\
$2$ \\
$9$ \\
$8$ \\
$7$ \\
\end{tabular}
\begin{tabular}{l}
\hspace{4.75em}\textbf{88} \\
$\begin{pmatrix}
1 & 0 & 0 & 0 \\
0 & -1 & 0 & 1 \\
0 & 0 & 1 & 0 \\
0 & 0 & 0 & 1 \\
-1 & 0 & 0 & 1 \\
0 & 1 & 0 & 0 \\
0 & 1 & -1 & -1 \\
0 & -1 & 0 & 0 \\
0 & 1 & 0 & -1 \\
\end{pmatrix} 
\begin{pmatrix}
-1 & 0 & 0 & 0\\
0 & 1 & 0 & -1\\
0 & 0 & 1 & 0\\
0 & 0 & 0 & -1\end{pmatrix}
=\begin{pmatrix}
-1 & 0 & 0 & 0 \\
0 & -1 & 0 & 0 \\
0 & 0 & 1 & 0 \\
0 & 0 & 0 & -1 \\
1 & 0 & 0 & -1 \\
0 & 1 & 0 & -1 \\
0 & 1 & -1 & 0 \\
0 & -1 & 0 & 1 \\
0 & 1 & 0 & 0 \\
\end{pmatrix} 
$
\end{tabular}
\end{tabular}
\end{table}
}

$\bullet$ 81 and 103
{\tiny 
\begin{table}[H]
\centering
\begin{tabular}{ccccccccc}
\begin{tabular}{c}
\\
$1$ \\
$2$ \\
$3$ \\
$4$ \\
$5$ \\
$6$ \\
$7$ \\
$8$ \\
$9$ \\
\end{tabular}
\begin{tabular}{c}
\textbf{81} \\
$\begin{pmatrix}
1 & 0 & 0 & 0 \\
0 & 1 & 0 & 0 \\
0 & 0 & 1 & 0 \\
0 & 0 & 0 & 1 \\
-1 & 0 & 0 & 1 \\
0 & -1 & 0 & 1 \\
0 & 1 & -1 & 0 \\
1 & 0 & 0 & -1 \\
0 & 0 & 0 & -1 \\
\end{pmatrix}$
\end{tabular}
& & & 
\begin{tabular}{c}
\\
$1$ \\
$6$ \\
$3$ \\
$4$ \\
$5$ \\
$2$ \\
$9$ \\
$8$ \\
$7$ \\
\end{tabular}
\begin{tabular}{l}
\hspace{4.5em}\textbf{103} \\
$\begin{pmatrix}
0 & 0 & 1 & 0 \\
0 & 0 & 0 & 1 \\
1 & 0 & 0 & 0 \\
0 & -1 & 1 & 0 \\
0 & -1 & 0 & 0 \\
0 & 1 & -1 & -1 \\
-1 & 0 & 0 & 1 \\
0 & 1 & 0 & 0 \\
0 & 1 & -1 & 0 \\
\end{pmatrix} 
\begin{pmatrix}
0 & 0 & 1 & 0\\
-1 & 0 & 0 & 1\\
-1 & 0 & 0 & 0\\
0 & 1 & 0 & 0\end{pmatrix}
=\begin{pmatrix}
-1 & 0 & 0 & 0 \\
0 & 1 & 0 & 0 \\
0 & 0 & 1 & 0 \\
0 & 0 & 0 & -1 \\
1 & 0 & 0 & -1 \\
0 & -1 & 0 & 1 \\
0 & 1 & -1 & 0 \\
-1 & 0 & 0 & 1 \\
0 & 0 & 0 & 1 \\
\end{pmatrix} 
$
\end{tabular}
\end{tabular}
\end{table}
}

$\bullet$ 82 and 91 and 107
{\tiny 
\begin{table}[H]
\centering
\begin{tabular}{ccccccccc}
\begin{tabular}{c}
\\
$1$ \\
$2$ \\
$3$ \\
$4$ \\
$5$ \\
$6$ \\
$7$ \\
$8$ \\
$9$ \\
\end{tabular}
\begin{tabular}{c}
\textbf{82} \\
$\begin{pmatrix}
1 & 0 & 0 & 0 \\
0 & 1 & 0 & 0 \\
0 & 0 & 1 & 0 \\
0 & 0 & 0 & 1 \\
-1 & 0 & 0 & 1 \\
0 & -1 & 0 & 1 \\
0 & 1 & -1 & 0 \\
0 & -1 & 0 & 0 \\
0 & 0 & 0 & -1 \\
\end{pmatrix}$
\end{tabular}
& & & 
\begin{tabular}{c}
\\
$1$ \\
$6$ \\
$3$ \\
$4$ \\
$5$ \\
$2$ \\
$9$ \\
$7$ \\
$8$ \\
\end{tabular}
\begin{tabular}{l}
\hspace{4.75em}\textbf{91} \\
$\begin{pmatrix}
1 & 0 & 0 & 0 \\
0 & -1 & 0 & 1 \\
0 & 0 & 1 & 0 \\
0 & 0 & 0 & 1 \\
-1 & 0 & 0 & 1 \\
0 & 1 & 0 & 0 \\
0 & 1 & -1 & -1 \\
0 & 1 & 0 & -1 \\
0 & 0 & 0 & -1 \\
\end{pmatrix} 
\begin{pmatrix}
-1 & 0 & 0 & 0\\
0 & 1 & 0 & -1\\
0 & 0 & 1 & 0\\
0 & 0 & 0 & -1\end{pmatrix}
=\begin{pmatrix}
-1 & 0 & 0 & 0 \\
0 & -1 & 0 & 0 \\
0 & 0 & 1 & 0 \\
0 & 0 & 0 & -1 \\
1 & 0 & 0 & -1 \\
0 & 1 & 0 & -1 \\
0 & 1 & -1 & 0 \\
0 & 1 & 0 & 0 \\
0 & 0 & 0 & 1 \\
\end{pmatrix} 
$
\end{tabular}
\end{tabular}
\end{table}
}

{\tiny 
\begin{table}[H]
\centering
\begin{tabular}{ccccccccc}
& & & 
\begin{tabular}{c}
\\
$1$ \\
$3$ \\
$2$ \\
$8$ \\
$5$ \\
$9$ \\
$6$ \\
$7$ \\
$4$ \\
\end{tabular}
\begin{tabular}{l}
\hspace{4.5em}\textbf{107} \\
$\begin{pmatrix}
1 & 0 & 0 & 0 \\
0 & 0 & 1 & 0 \\
0 & 1 & 0 & 0 \\
0 & 0 & 0 & -1 \\
-1 & 0 & 0 & 1 \\
0 & 0 & -1 & -1 \\
0 & -1 & 1 & 0 \\
0 & 0 & -1 & 0 \\
0 & 0 & 0 & 1 \\
\end{pmatrix} 
\begin{pmatrix}
-1 & 0 & 0 & 0\\
0 & 0 & -1 & 0\\
0 & -1 & 0 & 0\\
0 & 0 & 0 & 1\end{pmatrix}
=\begin{pmatrix}
1 & 0 & 0 & 0 \\
0 & -1 & 0 & 0 \\
0 & 0 & -1 & 0 \\
0 & 0 & 0 & -1 \\
-1 & 0 & 0 & 1 \\
0 & 1 & 0 & -1 \\
0 & -1 & 1 & 0 \\
0 & 1 & 0 & 0 \\
0 & 0 & 0 & 1 \\
\end{pmatrix} 
$
\end{tabular}
\end{tabular}
\end{table}
}

$\bullet$ 90 and 113
{\tiny 
\begin{table}[H]
\centering
\begin{tabular}{ccccccccc}
\begin{tabular}{c}
\\
$1$ \\
$2$ \\
$3$ \\
$4$ \\
$5$ \\
$6$ \\
$7$ \\
$8$ \\
$9$ \\
\end{tabular}
\begin{tabular}{c}
\textbf{90} \\
$\begin{pmatrix}
1 & 0 & 0 & 0 \\
0 & 1 & 0 & 0 \\
0 & 0 & 1 & 0 \\
0 & 0 & 0 & 1 \\
-1 & 0 & 0 & 1 \\
0 & -1 & 0 & 1 \\
0 & 1 & 0 & -1 \\
0 & 0 & -1 & 0 \\
0 & 0 & 0 & -1 \\
\end{pmatrix}$
\end{tabular}
& & & 
\begin{tabular}{c}
\\
$1$ \\
$6$ \\
$2$ \\
$9$ \\
$5$ \\
$8$ \\
$3$ \\
$7$ \\
$4$ \\
\end{tabular}
\begin{tabular}{l}
\hspace{4.5em}\textbf{113} \\
$\begin{pmatrix}
1 & 0 & 0 & 0 \\
0 & 0 & 1 & -1 \\
0 & 1 & 0 & 0 \\
0 & 0 & 0 & -1 \\
-1 & 0 & 0 & 1 \\
0 & 0 & -1 & 0 \\
0 & 0 & 1 & 0 \\
0 & -1 & 0 & 0 \\
0 & 0 & 0 & 1 \\
\end{pmatrix} 
\begin{pmatrix}
1 & 0 & 0 & 0\\
0 & 0 & 1 & 0\\
0 & -1 & 0 & 1\\
0 & 0 & 0 & 1\end{pmatrix}
=\begin{pmatrix}
1 & 0 & 0 & 0 \\
0 & -1 & 0 & 0 \\
0 & 0 & 1 & 0 \\
0 & 0 & 0 & -1 \\
-1 & 0 & 0 & 1 \\
0 & 1 & 0 & -1 \\
0 & -1 & 0 & 1 \\
0 & 0 & -1 & 0 \\
0 & 0 & 0 & 1 \\
\end{pmatrix} 
$
\end{tabular}
\end{tabular}
\end{table}
}


\section{$c_1$-preserving cohomology ring isomorphism}\label{sec:c1}

As we observed, cohomology rings do not distinguish toric Fano manifolds as varieties.  Very recently, motivated by McDuff's question on the uniqueness of toric actions on a monotone symplectic manifold, Y. Cho, E. Lee, S. Park and the third author made the following conjecture and verified it for Fano Bott manifolds (\cite{CLMP20}). 

\begin{conj}[\cite{CLMP20}]
If there is a cohomology ring  isomorphism between toric Fano manifolds which preserves their first Chern classes, then they are isomorphic as varieties.
\end{conj}

In this section, we prove the following theorem mentioned in Introduction, which gives further supporting evidence to the conjecture.  

\begin{thm}
The conjecture above is true for toric Fano $d$-folds with $d=3,4$ or with Picard number $\ge 2d-2$.  
\end{thm} 

In order to prove the theorem above, it suffices to check that there is no $c_1$-preserving cohomology ring isomorphism between toric Fano $d$-folds which have isomorphic cohomology rings.  
If there is a $c_1$-preserving cohomology ring  isomorphism between toric Fano $d$-folds $X$ and $Y$, then $c_1(X)^d$ evaluated on the fundamental class of $X$, in other words the degree $(-K_X)^d$ of $X$, agrees with that of $Y$. We obtain the following tables from the database of {\O}bro. They together with Tables~\ref{table:diff3fold} and~\ref{table:diff4fold} show that the degrees are different for toric Fano $3$- or $4$-folds which have isomorphic cohomology rings except one pair ID 70 and 141.  

{\scriptsize 
\begin{table}[H]
\centering
\begin{tabular}{|c|c|c|c|c|c|c|c|} \hline
ID & 11, 18 & 10, 13  \\ \hline
degree & 52, 44 & 44, 40  \\ \hline 
\end{tabular} 
\medskip
\caption{Degrees of toric Fano $3$-folds with isomorphic cohomology rings} 
\label{table:degree3fold}
\end{table}

\begin{table}[H]
\centering
\begin{tabular}{|c|c|c|c|c|c|c|c|} \hline
ID &  70, 141 & 30, 43 & 68, 134 & 129, 136 & 28, 32 & 67, 118 & 123, 125 \\ \hline
degree &  513, 513 & 592, 400 & 432, 480 & 496, 400 & 478, 382 & 351, 447 & 415, 367 \\ \hline 
\end{tabular} 

\bigskip
\begin{tabular}{|c|c|c|c|c|c|c|c|} \hline
ID &  74, 96&83, 108&95, 131&29, 39&111, 116&50, 57 & 73, 76, 92 \\ \hline
degree &480, 352& 448, 352& 416, 352&463, 337&389, 347&417, 369& 394, 330, 310 \\ \hline 
\end{tabular} 

\bigskip
\begin{tabular}{|c|c|c|c|c|c|c|c|} \hline
ID &  77, 88&81, 103&82, 91, 107&90, 113&72, 87&78, 86 \\ \hline
degree & 405, 331&373, 325&341, 363, 229&352, 320&308, 268&298, 278 \\ \hline 
\end{tabular} 
\medskip
\caption{Degrees of toric Fano $4$-folds with isomorphic cohomology rings} 
\label{table:degree4fold}
\end{table}
}

As for ID 70 and 141, more detailed observation is necessary.  It follows from Table~\ref{tab:VMNF69} and Remark~\ref{rem:pont} that 
\begin{equation} \label{eq:70141}
\begin{split}
H^*(X_{70})&=\ZZ[x,y]/(x^3, y(x-y)^2),\qquad c_1(X_{70})=x+3y\\
H^*(X_{141})&=\ZZ[x,y]/(x^3, y^2(x-y)),\qquad c_1(X_{141})=2x+3y.
\end{split}
\end{equation}
An elementary computation shows that an isomorphism $H^*(X_{70})\to H^*(X_{141})$ is given by either $(x,y)\to (x,x-y)$ or $(x,y)\to (-x,-x+y)$ but both isomorphisms are not $c_1$-preserving.  This completes the proof of the theorem when $d=3,4$. 

As investigated in Section~\ref{sec:Picard}, toric Fano $d$-folds with Picard number $\ge 2d-2$ which have isomorphic cohomology rings are the following three pairs:
\begin{enumerate}
\item $Z_3^d=Z_3\times (X_{P_6})^{\frac{d-3}{2}}$ and $Z_4^d=Z_4\times (X_{P_6})^{\frac{d-3}{2}}$, where $d$ is odd $\ge  3$
\item $W_5^d=W_5\times  (X_{P_6})^{\frac{d-4}{2}}$ and $W_7^d=W_7\times  (X_{P_6})^{\frac{d-4}{2}}$, where $d$ is even $\ge  4$.   
\item $W_6^d=W_6\times  (X_{P_6})^{\frac{d-4}{2}}$ and $W_8^d=W_8\times  (X_{P_6})^{\frac{d-4}{2}}$, where $d$ is even $\ge  4$,
\end{enumerate}
see \eqref{eq:Zid} and \eqref{eq:Wid}.  Here 
\[
(Z_3,Z_4)=(X_{10},X_{13}),\quad (W_5,W_7)=(X_{72},X_{87}),\quad (W_6,W_8)=(X_{78},X_{86})
\]
as mentioned above Convention in Sections~\ref{sec:dim3} and~\ref{sec:dim4}.  The degree of a product of projective varieties $X$ and $Y$ of dimension $p$ and $q$ is $\binom{p+q}{p}$ times the product of degrees of $X$ and $Y$, so it follows from Tables~\ref{table:degree3fold} and~\ref{table:degree4fold} that the three pairs above have different degrees respectively.  This completes the proof of the theorem.

\end{document}